\documentclass[10pt,reqno]{amsart}
\usepackage[utf8]{inputenc}
\usepackage[english]{babel}
\usepackage{amsmath, amssymb, amsthm, comment}
\usepackage{mathrsfs}
\usepackage{enumitem}
\usepackage{graphicx}
\usepackage{tikz-cd}
\usepackage{xcolor}
\numberwithin{equation}{section} %para numerar ecuaciones por sección

\usepackage{diagbox} % to have the diagonal separation on the 1,1 entry on a table
\usepackage{tensor}  %para escribir tensores y que los superíndices no queden debajo de los subíndices

\usepackage{hyperref} %this have to be load BEFORE amsrefs

\usepackage[alphabetic]{amsrefs} %alph is for references as [NameYear]
\theoremstyle{plain}
\newtheorem{thm}{Theorem}[section]
\newtheorem{prop}[thm]{Proposition}
\newtheorem{lemma}[thm]{Lemma}

\theoremstyle{definition}

\allowdisplaybreaks

\newcommand{\C}{\mathbb{C}}

\def\XXint#1#2#3{{\setbox0=\hbox{$#1{#2#3}{\int}$ }
\vcenter{\hbox{$#2#3$ }}\kern-.6\wd0}}

\usepackage{scalerel,stackengine}
\stackMath
\newcommand\hhat[1]{%
\savestack{\tmpbox}{\stretchto{%
  \scaleto{%
    \scalerel*[\widthof{\ensuremath{#1}}]{\kern.1pt\mathchar"0362\kern.1pt}%
    {\rule{0ex}{\textheight}}%WIDTH-LIMITED CIRCUMFLEX
  }{\textheight}% 
}{2.4ex}}%
\stackon[-6.9pt]{#1}{\tmpbox}%
}
\newcommand{\s}{\mathcal{S}}

\title[Partial data Calderon for quasilinear conductivities]{Partial data Calder\'{o}n problem for quasilinear conductivities in dimension 2}

\author{Ruirui Wu}
\author{Ting Zhou}

\begin{document}

\begin{abstract}
In this paper, we prove a uniqueness result for the partial data Calder\'{o}n problem with quasilinear conductivity in two dimensions. The proof is based on higher-order linearization and the use of CGO solutions in dimension two that vanish on part of the boundary. Since derivatives of the solutions appear in the integral identity, we need improved remainder estimates, for which we introduce a modification in the choice of the phase, analogous to limiting Carleman weights. We also analyze how combinations of phases produce specific patterns in the products of solutions, which allows us to apply both stationary and nonstationary phase arguments to recover the conductivity.
\end{abstract}

\maketitle

\section{Introduction}

Let $\Omega \subset \mathbb{R}^2$ be a connected bounded open set with $C^\infty$ boundary. 
We consider the boundary value problem for the following conductivity equation with isotropic quasilinear conducticity:
\begin{equation}
\begin{cases}
\operatorname{div}\!\left(\gamma(x,u,\nabla u)\,\nabla u\right)=0 & \text{in }\Omega,\\[4pt]
u=f & \text{on }\partial\Omega.
\end{cases}
\tag{1.1}
\end{equation}

Here, we assume that the function 
\[
\gamma:\overline{\Omega}\times \mathbb{C}\times \mathbb{C}^n\to\mathbb{C}
\]
satisfies the following conditions:

\begin{itemize}
\item[(i)] for each $(x,z)\in \overline{\Omega}\times\mathbb{C}$, the map 
\(\tau\mapsto \gamma(x,\tau,z)\) is holomorphic with values in the Hölder space 
\(C^{1,\alpha}(\Omega)\) for some \(0<\alpha<1\);

\item[(ii)] \(\gamma(x,\tau,0)=1\), for all \(x\in \Omega\) and all \(\tau\in \mathbb{C}\).
\end{itemize}

Under the assumptions (i) and (ii), there exist \(\delta>0\) and \(C>0\) such that 
when 
\[
f\in B_{\delta}(\partial\Omega):=\{f\in C^{2,\alpha}(\partial\Omega): 
\|f\|_{C^{2,\alpha}(\partial\Omega)}<\delta\},
\]
the problem (1.1) has a unique solution 
\(u=u_{f}\in C^{2,\alpha}(\overline{\Omega})\). 
This local well-posedness of the forward problem is given in \cite{CFKKU2021calderon}, Appendix B.

Let \(\Gamma\subset \partial\Omega\) be an arbitrary non-empty open subset of the boundary. 
Associated to the problem (1.1), we define the partial Dirichlet-to–Neumann map
\[
\Lambda^\Gamma_\gamma(f)
=\left(\gamma(x,u,\nabla u)\,\partial_\nu u\right)\big|_{\Gamma},
\]
where \(f\in B_{\delta}(\partial\Omega)\) with \(\operatorname{supp}(f)\subset\Gamma\) 
. Here \(\nu\) is the unit outer normal to the boundary.

In this work, we prove the following uniqueness theorem for the partial data:
\begin{thm}
Let $\Omega\subset\mathbb{R}^2$ be a connected bounded open set with $C^\infty$
boundary, and let \(\Gamma\subset\partial\Omega\) be an arbitrary open non-empty subset of 
the boundary. Assume that 
\[
\gamma_1,\gamma_2:\overline{\Omega}\times\mathbb{C}\times\mathbb{C}^n\to\mathbb{C}
\]
satisfy assumptions (i) and (ii). If  
$\gamma_1, \gamma_2 $ agrees to infinite order on the boundary, and
\[
\Lambda^\Gamma_{\gamma_1}(f)=\Lambda^\Gamma_{\gamma_2}(f),
\qquad 
\forall f\in B_{\delta}(\partial\Omega),\ 
\operatorname{supp}(f)\subset\Gamma,
\]
Then we have
\[
\gamma_1=\gamma_2 \quad\text{in } \overline{\Omega}\times\mathbb{C}\times\mathbb{C}^n.
\] 
\end{thm}
 The proof of Theorem 1.1  reduces to the following completeness result:
 
\begin{prop}\label{prop:main1}
Let $\Omega \subset \mathbb{R}^2$, be a bounded open domain with $C^{\infty}$ boundary. Let $\gamma_0 \in C^{\infty}(\bar{\Omega})$ and assume that $\gamma_0$ satisfies the assumption (i). Let $m \in \mathbb{N}$ and let $T$ be a $C^\infty$ smooth function with values in the space of symmetric tensors of rank $m$. Assume also that $T$ vanishes to infinite order on $\partial \Omega$. Suppose that 
\begin{multline}\label{prop1equ}
%\begin{array}{r}
\sum_{\left(l_1, \ldots, l_{m+1}\right) \in \pi(m+1)} \sum_{j_1, \ldots, j_m=0}^2 \int_{\Omega} T^{j_1 \ldots j_m}(x)\left(u_{l_1}, \nabla u_{l_1}\right)_{j_1} \ldots\left(u_{l_m}, \nabla u_{l_m}\right)_{j_m} \\
\times\nabla u_{l_{m+1}} \cdot \nabla u_{m+2} \s d x=0.
%\end{array}
\end{multline}
for all $u_l \in C^{\infty}(\bar{\Omega})$ solving $\nabla \cdot\left(\gamma_0 \nabla u_l\right)=0$ in $\Omega$, and $supp (u)|_{\partial \Omega}\subset \Gamma$, $ l=1, \ldots, m+2$. Then $T$ vanishes identically on $\Omega$.
Here $\left(u_l, \nabla u_l\right)_j, j=0,1, 2$ stands for the $j^{\textrm{th}}$ component of the vector $\left(u_l, \partial_{x_1} u_l, \ldots, \partial_{x_n} u_l\right)$, and $\pi(m+1)$ stands for the set of all distinct permutations of $\{1,\ldots, m +1\}$. 
\end{prop}

\subsection{Motivation and previous literature}

We briefly review earlier related works. For the linear conductivity equation, uniqueness for smooth conductivities in dimensions $n \geq 3$ was established by Sylvester and Uhlmann \cite{SyU} and Novikov \cite{Nov}, while in dimension two, Nachman \cite{Nach} proved uniqueness for $C^2$ conductivities. Subsequent works have relaxed these regularity assumptions: in dimensions $n \geq 3$, uniqueness holds for $W^{1,\infty}(\Omega)$ conductivities by Haberman and Tataru \cite{W1inf}, and in dimension two, for $L^\infty$ conductivities by Astala and P\"{a}iv\"{a}rinta \cite{AP}.

The Calderón inverse problem with partial boundary data has been studied in the linear setting. In dimensions $n\geq 3$, uniqueness from partial measurements was established in \cite{BU02} and \cite{KSU07}, while in dimension two it was proved by \cite{IUY10} and \cite{IUY11} using complex analytic methods. The case on manifold was discussed in \cite{GuillarmouTzou2011}. These results rely on the construction of special solutions vanishing on parts of the boundary and on density properties for products of solutions.

	% Following the foundational works \cite{KLU2018, FEIZMOHAMMADI20204683,LASSAS202144}, the literature on inverse problems for nonlinear equations based on the higher-order linearization method has grown rapidly. Earlier studies on inverse problems for quasilinear conductivity equations include  \cite{hervas2002inverse, egger2014simultaneous, munoz2020calderon,shankar2020recovering}.
	% More recently, works such as  \cite{LLLS2019partial,LLST2022inverse,KU2019partial,KU2019remark,feizmohammadi2023inverse, liimatainen2024uniqueness} have investigated inverse problems for semilinear elliptic equations with general nonlinearities, including cases with partial data.  

    For the quasilinear conductivity equation with $\gamma(x,u,\nabla u)$, uniqueness in dimensions $n \geq 3$ was established in \cite{CFKKU2021calderon}, while in dimension two it was proved in \cite{LW24} using a modification of the solutions in \cite{BU02}. In addition, inverse problems for the minimal surface equation have been studied in \cites{carstea2024inverse,nurminen1,nurminen2,minimal_surface_general} on Riemannian surfaces and in Euclidean domains.

The partial data problem for quasilinear conductivities is less developed. While higher-order linearization has proved effective for nonlinear inverse problems with full data, its application to partial-data settings introduces additional difficulties. In particular, one must combine higher-order linearization with the construction of solutions that vanish on inaccessible boundary portions. A uniqueness result for quasilinear conductivity equations with the gradient of the solution restricted to be a fixed direction was obtained in \cite{KKU}, where the authors use higher-order linearization together with density properties for products of gradients of harmonic functions vanishing on part of the boundary. The density result was proved from local to global following the idea in \cite{DKSU09}.

In our work, we will remove the restriction on the direction of $\nabla u$, by using the CGO solutions constructed by the Carleman estimate method as in \cite{IUY10} and \cite{GuillarmouTzou2011}, and higher-order linearization. 

\section{CGO Solutions}\label{section:CGO}
 Using the identity 
\begin{equation}\label{eq1}
    -\nabla \cdot \gamma \nabla\left(\gamma^{-1 / 2} u\right)=\gamma^{1 / 2}(-\Delta+q) u,
\end{equation}
where 
\[
q=\frac{\Delta \sqrt{\gamma}}{\sqrt{\gamma}},
\] 
we can obtain solutions to the conductivity equation from solutions to the Schrödinger equation $(-\Delta +V)u = 0$.

Let $\Gamma_0: = \partial \Omega \setminus \Gamma$ be the set where we do not know the Dirichlet-to-Neumann map. Let $\Phi =\varphi +i\psi$ be a holomorphic function that is purely real on $\Gamma_0$ and Morse on $\Omega$, with a critical point at $z_0\in \Omega$.

Consider the CGO solutions as described in \cite{IUY10} and \cite{GuillarmouTzou2011}. For completeness, we include a sketch of the construction and properties of these solutions. Let
\begin{align}\label{CGO_solution}
u &= \frac{1}{\sqrt{\gamma}} (e^{\Phi/h}(a + h a_0 + r_1)
      + \overline{e^{\Phi/h}(a + h a_0 + r_1)}
      + e^{\varphi/h} r_2)
\end{align}
for sufficiently small \(h>0\), where \(a\) is holomorphic, \(u\in C^k(\Omega)\) for some sufficiently large \(k\in\mathbb N\), and \(a_0\in H^2(\Omega)\) is holomorphic. We further assume that \(a(z_0)\neq 0\) and that \(a\) vanishes to high order at all other critical points \(z'\in \Omega\) of \(\Phi\). Moreover, \(a\) is chosen to be purely imaginary on \(\Gamma_0\). The remainder $r_1$ is constructed to satisfy 
\begin{align}\label{r1 construction}
e^{-\Phi/h}(\Delta+q)e^{\Phi/h}(a+r_{1})
=
O_{L^{2}}(h|\log h|),
\end{align} and has the form $r_1 = r_{11}+hr_{12}$.

Now consider $b: = \partial G(qa)-\omega = b(z)dz$, where $G$ is the Green operator of the Laplacian on the smooth domain $\Omega$
with boundary with Dirichlet condition, and $\omega$ is a smooth holomorphic $1$-form on $\Omega$.

We consider a local coordinates $z$ centered at a critical point $z'$ of $\Phi$
(i.e.\ $z'=\{z=0\}$ in this coordinate) , and construct $\omega$ such that at each critical point $z'$ of $\Phi$, we have
\begin{equation}\label{property of b}
\left|\partial_{\bar z}^{m}\partial_{z}^{\ell} b(z)\right|
= O\!\left(|z|^{2+\alpha-\ell-m}\right), \qquad \ell+m\le2,
\quad \text{if } z'\neq z_0,
\end{equation}
\[
|b(z)| = O(|z|), \qquad \text{if } z'=z_0 .
\]

Now, we let $\chi_1\in C_0^\infty(\Omega)$ be a cutoff function supported in a small neighbourhood $U_{z_0}$ of
the critical point $z_0$ and identically $1$ near $z_0$, and $\chi\in C_0^\infty(\Omega)$ is defined similarly
 with $\chi=1$ on the support of $\chi_1$.
%We will construct $r_1=r_{11}+hr_{12}$ in two steps: first,
% we will construct $r_{11}$ to solve equation (11) locally near the critical point $p$ of $\Phi$
% and then we will construct the global correction term $r_{12}$ away from $p$ by using the extra
% vanishing of $b$ in (12) at the other critical points.

We define locally in complex coordinates centered at $p$ and containing the support of $\chi$

\[
r_{11} := \chi e^{-2i\psi/h} R\big(e^{2i\psi/h}\chi_1 b\big)
\]
where
\[
Rf(z) := -(2\pi i)^{-1}\int_{\mathbb{R}^2}
\frac{1}{\bar z-\bar \xi}\,f\,d\bar\xi \wedge d\xi
\]
for $f\in L^\infty$ compactly supported. Note that $R$ is the classical Cauchy--Riemann operator
inverting $\partial_z$ locally. Extend $r_{11}$ by $0$ outside the neighbourhood of $z_0$.
The function $r_{11}$ is in $C^{3+\alpha}(\Omega)$ and we have
\begin{equation}
e^{-2i\psi/h}\partial\!\left(e^{2i\psi/h} r_{11}\right)
= \chi_1(-\partial G(qa)+\omega) + \eta
\end{equation}
with
\[
\eta := e^{-2i\psi/h}R(e^{2i\psi/h}\chi_1 b)\,\partial\chi .
\]

We then construct $r_{12}$ by observing that $b$ vanishes to order $2+\alpha$ at
critical points of $\Phi$ other than $z_0$, as in \eqref{property of b}, and $\partial\chi=0$
in a neighbourhood of any critical point of $\psi$, so we can find $r_{12}$ satisfying
\[
2i r_{12}\partial\psi = (1-\chi_1)b .
\]
Next, by Lemma 4.1.2 in \cite{GuillarmouTzou2011}, we have the following estimates.

\begin{lemma}\label{r_1: L2 est}
\[
\|\eta\|_{H^2} = O(|\log h|), \qquad
\|\eta\|_{H^1} \le O(h|\log h|), \qquad
\|r_1\|_{L^2} = O(h), \qquad
\|r_1 - h\tilde r_{12}\|_{L^2} = o(h).
\]
where $\widetilde{r}_{12}$ solves 
\[
2i \widetilde{r}_{12}\partial\psi = b .
\]
\end{lemma}

We also prove a $H^2$ estimate for $r_1$: 
\begin{lemma}\label{r1: H^2 est original}
$$||r_1||_{H^2(\Omega)} = o(\frac{1}{h}).$$
\end{lemma}
\begin{proof}
Now we have
\begin{equation}
\|r_1\|_{H^2(\Omega)}
\le
\left\|
\chi e^{-2i\psi/h} R\!\left(e^{2i\psi/h}\chi_1 b\right)
-
h\,\frac{\chi_1 b}{2\partial_z\psi}
\right\|_{H^2(U_{z_0})}
+
h\|\tilde r_{12}\|_{H^2(\Omega)}
\end{equation}
By \cite{IUY10} Proposition 2.7, we have
\begin{align*}
&\left\|
\chi e^{-2i\psi/h} R\!\left(e^{2i\psi/h}\chi_1 b\right)
-
h\,\frac{\chi_1 b}{2\partial_z\psi}
\right\|_{L^2(U_{z_0})} \\ &= \left\|\chi \left(e^{-2i\psi/h} R\!\left(e^{2i\psi/h}\chi_1 b\right)
-
h\,\frac{\chi_1 b}{2\partial_z\psi}\right)
\right\|_{L^2(U_{z_0})} = o(h)
\end{align*}

By the form of 
$$
 e^{{-2i\psi}/h} R e^{2i\psi/h} {\partial}\left(\frac{\chi_1 b}{2{\partial_z} \psi}\right),
$$
we see that the same estimates hold with a loss of $h^{2}$
for any derivatives of order less than or equal to $2$. So we have

$$
\left\|
\chi e^{-2i\psi/h} R\!\left(e^{2i\psi/h}\chi_1 b\right)
-
h\,\frac{\chi_1 b}{2\partial_z\psi}
\right\|_{H^2(U_{z_0})}=o(h^{-1})
$$

Now since $\widetilde{r}_{12}$ is independent of $h$, we have
$$||r_1||_{H^2(\Omega)} = o(h^{-1}).$$
\end{proof}

From above, we have shown
 there exists
\[
r_{1}\in H^{2}(\Omega)
\]
such that
\[
\|r_{1}\|_{L^{2}}=O(h)
\]
and satisfies \eqref{r1 construction}.

In the next step, it can be shown that there exists a holomorphic function $a_0 \in H^2(\Omega)$ independent of $h$ such that the solution vanishes on $\Gamma_0$. The previous requirement that $\Phi$ is purely real on $\Gamma_0$, $a$ is purely imaginary on $\Gamma_0$ is used here.

Finally, the following result is shown by a Carleman estimate in \cite{GuillarmouTzou2011}.

\begin{lemma}\label{r_2: L2 est}
    There exists $r_2$, satisfying $||r_2||_{L^2} = O(h^{3/2}|\log h|)$, such that \eqref{CGO_solution} solves the Schr\"odinger equation.
\end{lemma}

% For convenience, we include the Carleman estimate used here.

% \begin{prop}
% Let
% $\varphi:\Omega\to\mathbb{R}$ be a $C^k(\Omega)$ harmonic Morse function for
% $k$ large. Then for all $V\in L^\infty(\Omega)$ there exist constants
% $C>0$ and $h_0>0$ such that, for all $h\in(0,h_0)$ and
% $u\in C^\infty(\Omega)$ with $u|_{\partial\Omega}=0$, we have
% \begin{equation}
% \frac{1}{h}\|u\|_{L^2(\Omega)}^2
% +\frac{1}{h^2}\|u|d\varphi|\|_{L^2(\Omega)}^2
% +\|du\|_{L^2(\Omega)}^2
% +\|\partial_\nu u\|_{L^2(\Gamma_0)}^2
% \le
% C\left(
% \|e^{-\varphi/h}(\Delta_g+V)e^{\varphi/h}u\|_{L^2(\Omega)}^2
% +\frac{1}{h}\|\partial_\nu u\|_{L^2(\Gamma)}^2
% \right).
% \end{equation}
% where $\partial_\nu$ is the exterior unit normal vector field to
% $\partial\Omega$.
% \end{prop}

We will also give the $H^1_{loc}$ and $H^2_{loc}$ estimates for $r_2$ for later use. We note that the local estimates are sufficient for later use, since we have assumed boundary determination.

\begin{lemma}\label{r_2: H^2 loc est}
\[
\|r_2\|_{H^1_{loc}} = O(h^{1/2}|\log h|), \qquad
\|r_2\|_{H^2_{loc}} = O(h^{-1/2}|\log h|).
\]    
\end{lemma}
\begin{proof}
First, we note that $r_2$ satisfies 
$$
e^{-\varphi/h}(\Delta+q)e^{\varphi/h}r_2 = -e^{-\varphi/h}(\Delta+q)\left(e^{\Phi/h}(a+r_1+ha_0) + \overline{e^{\Phi/h}(a+r_1+ha_0)}\right)=O_{L^2}(h\log h)
$$
We first show that for fixed $h>0$, $r_2 \in H^2_{loc}$. Denote 
\[
P_h:=e^{\varphi/h}(\Delta+q)e^{-\varphi/h}.
\]
Then by previous lemmas and the equation for $r_2$, we have
\[
r_2\in L^2(\Omega),
\qquad
P_h r_2=f:=-e^{-\varphi/h}(\Delta+q)\left(e^{\Phi/h}(a+r_1+ha_0) + \overline{e^{\Phi/h}(a+r_1+ha_0)}\right)\in L^2(\Omega).
\]
Consider a coordinate patch
$U\Subset \Omega$. Locally,
\[
P_h r_2
=
\Delta r_2
+\frac1h A(x)\cdot \nabla r_2
+\frac1{h^2}B(x)r_2
+\frac1h C(x)r_2
+q(x)r_2,
\]
where the coefficients are bounded functions on compact
subsets of $U$.

By the equation satisfied by $r_2$, we have in the sense of distributions
\[
\Delta r_2
=
f
-\frac1h A\cdot \nabla r_2
-\frac1{h^2}Br_2
-\frac1h Cr_2
-qr_2.
\]
Since we know $r_2\in L^2_{\mathrm{loc}}$, we have
$\nabla r_2\in H^{-1}_{\mathrm{loc}}$,
and hence
\[
A\cdot \nabla r_2\in H^{-1}_{\mathrm{loc}}.
\]
Also,
\[
Br_2,\ Cr_2,\ qr_2\in L^2_{\mathrm{loc}}
\subset H^{-1}_{\mathrm{loc}}.
\]
Thus
$
\Delta r_2\in H^{-1}_{\mathrm{loc}}.
$
By local elliptic regularity for $\Delta$, we obtain
\[
r_2\in H^1_{\mathrm{loc}}(U).
\]

Now return to the equation for $r_2$. Since $r_2\in H^1_{\mathrm{loc}}(U)$, we have
\[
A\cdot \nabla r_2\in L^2_{\mathrm{loc}}(U),
\]
and the zero-order terms still satisfy
\[
Br_2,\ Cr_2,\ qr_2\in L^2_{\mathrm{loc}}(U).
\]
Therefore
$
\Delta r_2
\in L^2_{\mathrm{loc}}(U).
$
Applying local elliptic regularity again, we get
\[
r_2\in H^2_{\mathrm{loc}}(U).
\]
Since $U\Subset \Omega$ was arbitrary, we conclude that
\[
r_2\in H^2_{\mathrm{loc}}(\Omega).
\]

Next, we establish a local Caccioppoli estimate and then bootstrap to local $H^1$ and $H^2$ estimates via elliptic regularity.

Now we have
\[
\|f\|_{L^2} = O(h|\log h|), \qquad \|r_2\|_{L^2} = O(h^{3/2}|\log h|).
\]
Fix nested open sets
\[
K \Subset U_0 \Subset U_1 \Subset \Omega.
\]

We begin with a Caccioppoli estimate on $U_0$. Let $\zeta \in C_c^\infty(U_1)$ with $\zeta \equiv 1$ on $U_0$. 

Multiplying $\Delta r_2$ by $\zeta^2 \overline{r_2}$, integrating over $U_1$, and integrating by parts gives
\[
\|\zeta \nabla r_2\|_{L^2}^2 \le C\|r_2\|_{L^2(U_1)}^2 + \left\|\int_{U_1} \zeta^2 (\Delta r_2)\overline{r_2}\right\|
\]
for some constant $C$.
Substituting the equation for $\Delta r_2$ and applying Cauchy–Schwarz to each term on the right yields
\[
\left\|\int_{U_1} \zeta^2 (\Delta r_2)\overline{r_2}\right\|
\le \|f\|_{L^2(U_1)}\|r_2\|_{L^2(U_1)}
+ \frac{C}{h}\|\zeta\nabla r_2\|_{L^2}\|r_2\|_{L^2(U_1)}
+ \frac{C}{h^2}\|r_2\|_{L^2(U_1)}^2.
\]
By Young's inequality,
\[
\frac{C}{h}\|\zeta\nabla r_2\|_{L^2}\|r_2\|_{L^2}
\le \frac12\|\zeta\nabla r_2\|_{L^2}^2 + \frac{C}{h^2}\|r_2\|_{L^2}^2
\]
for some constant $C$. The first term on the right absorbs into the left-hand side, giving
\[
\|\nabla r_2\|_{L^2(U_0)}^2
\le C\left( \frac1{h^2}\|r_2\|_{L^2(U_1)}^2 + \|f\|_{L^2(U_1)}\|r_2\|_{L^2(U_1)} \right),
\]
hence
\[
\|\nabla r_2\|_{L^2(U_0)}
\le C\left( \frac1h\|r_2\|_{L^2(U_1)} + \|f\|_{L^2(U_1)}^{1/2}\|r_2\|_{L^2(U_1)}^{1/2} \right).
\]
With the assumed bounds,
\[
\frac1h\|r_2\|_{L^2} = O(h^{1/2}|\log h|), \qquad
\|f\|_{L^2}^{1/2}\|r_2\|_{L^2}^{1/2} = O(h^{5/4}|\log h|),
\]
and the second term is subdominant. Therefore
\[
\|r_2\|_{H^1(U_0)} = O(h^{1/2}|\log h|).
\]

Since $K \Subset U_0$, interior elliptic regularity for the Laplacian gives
\[
\|r_2\|_{H^2(K)} \le C\left( \|\Delta r_2\|_{L^2(U_0)} + \|r_2\|_{L^2(U_0)} \right).
\]
The equation gives
\[
\|\Delta r_2\|_{L^2(U_0)} \le \|f\|_{L^2(U_0)} + \frac{C}{h}\|\nabla r_2\|_{L^2(U_0)} + \frac{C}{h^2}\|r_2\|_{L^2(U_0)}.
\]
Using $\|f\|_{L^2} = O(h|\log h|)$, $
\frac1h\|\nabla r_2\|_{L^2(U_0)} = O(h^{-1/2}|\log h|)
$, and $
\frac1{h^2}\|r_2\|_{L^2(U_0)} = O(h^{-1/2}|\log h|)
$, 
we obtain
\[
\|\Delta r_2\|_{L^2(U_0)} = O(h^{-1/2}|\log h|),
\]
and consequently
\[
\|r_2\|_{H^2(K)} = O(h^{-1/2}|\log h|).
\]

% \color{blue}
% For the $H^1$ estimate, we follow the same proof of Proposition 4.1 in \cite{GuillarmouTzou2011}, but in the solvability result Lemma 4.3.2, instead of the $L^2$ norm, consider the norm
% $$\|u\|_{H^1_h}: = \frac{1}{h}\|u\|_{L^2(\Omega)}+\|du\|_{L^2(\Omega)}.$$
\end{proof}

Next, we consider the remainder estimates for the case where $\Phi$ has no critical points.  

\begin{lemma}\label{r_1: H2 est new}
    If $\Phi$ has no critical points, then 
    $$
||r_{1}||_{H^2} = O(h).
$$
\end{lemma}
\begin{proof}
Note that in this case
$$
R e^{{2i\psi}/h} f=h\left[e^{2{i}\psi /h} \frac{f}{\bar{\partial} \psi}+ R \left(e^{2i\psi/h} \bar{\partial}\left(\frac{f}{\bar{\partial} \psi}\right)\right)\right]
$$
 holds for any $f \in C_0^1(\bar{\Omega})$.

Thus
$$
e^{{-2i\psi}/h}R e^{{2i\psi}/h} f=h\left[ \frac{f}{\bar{\partial} \psi}+ e^{{-2i\psi}/h} R \left(e^{2i\psi/h} \bar{\partial}\left(\frac{f}{\bar{\partial} \psi}\right)\right)\right]
$$

 If $f$ has better regularity ($C_0^2(\bar{\Omega})$ for example), and $\psi$ has no critical points, we can obtain the same identity for $R\left(e^{2i\psi/h} \bar{\partial}\left(\frac{f}{\bar{\partial} \psi}\right)\right)$:
 $$
R \left(e^{2i\psi/h} \bar{\partial}\left(\frac{f}{\bar{\partial} \psi}\right)\right) = 
h\left[e^{2i \psi/h} \frac{\bar{\partial}\left(\frac{f}{\bar{\partial} \psi}\right)}{\bar{\partial} \psi}+ R \left(e^{2i  \psi/h} \bar{\partial}\left(\frac{\bar{\partial}\left(\frac{f}{\bar{\partial} \psi}\right)}{\bar{\partial} \psi}\right)\right)\right].
 $$
Substituting back to the previous identity, we get 
$$
R e^{2i\psi/h} f=h\left(e^{2i \psi/h} \frac{f}{\bar{\partial} \psi}\right)+ h^2\left[e^{2i\psi/h} \frac{\bar{\partial}\left(\frac{f}{\bar{\partial} \psi}\right)}{\bar{\partial} \psi}+ R\left(e^{2i\psi/h} \bar{\partial}\left(\frac{\bar{\partial}\left(\frac{f}{\bar{\partial} \varphi}\right)}{\bar{\partial} \varphi}\right)\right)\right]
$$
Hence 
$$
e^{-2i\psi/h} R e^{2i\psi/h} f =h\left(\frac{f}{\bar{\partial} \psi}\right)+ h^2\left[\frac{\bar{\partial}\left(\frac{f}{\bar{\partial} \psi}\right)}{\bar{\partial} \psi}+ e^{-2i\psi/h}R\left(e^{2i \psi/h} \bar{\partial}\left(\frac{\bar{\partial}\left(\frac{f}{\bar{\partial} \psi}\right)}{\bar{\partial} \psi}\right)\right)\right]
$$
From here, we have
$$
||e^{-2i\psi/h} R e^{2i\psi/h} f -h\left(\frac{f}{\bar{\partial} \psi}\right)||_{H^1} = O(h)
$$
If we further assume $f\in C^3_0$, then similarly we have 
$$
||e^{-2i\psi/h} R e^{2i\psi/h} f -h\left(\frac{f}{\bar{\partial} \psi}\right)||_{H^2} = O(h)
$$
Applying this observation to 
$
\left\|
\chi e^{-2i\psi/h} R\!\left(e^{2i\psi/h}\chi_1 b\right)
-
h\,\frac{\chi_1 b}{2\partial_z\psi}
\right\|
$, we get
$$
\left\|
\chi e^{-2i\psi/h} R\!\left(e^{2i\psi/h}\chi_1 b\right)
-
h\,\frac{\chi_1 b}{2\partial_z\psi}
\right\|_{H^2(U_{z_0})} = O(h).
$$
Consequently, also using that $\widetilde{r}_{12}$ is independent of $h$, we have
$$
||r_{1}||_{H^2} = O(h).
$$
\end{proof}
Moreover, since
\[
\eta = e^{-2i\psi/h}R(e^{2i\psi/h}\chi_1 b)\,\partial\chi,
\]
if $\psi$ has no critical point, by the formula in the above proof,  we have 
$$
\|\eta\|_{L^2}=O(h).
$$
Hence, by the same proof as before, in this case we have the following better estimates for $r_2$:
\begin{lemma}\label{r_2:  est new}
\[
\|r_2\|_{L^2} = O(h^{3/2}), \qquad
\|r_2\|_{H^1_{loc}} = O(h^{1/2}), \qquad
\|r_2\|_{H^2_{loc}} = O(h^{-1/2}).
\]    
    
\end{lemma}

\section{Proof of Proposition: the case m=1}
Without loss of generality, we assume $\Omega$ is an open subset of the unit disk and $z_0=0$ is an interior point. Let $\Gamma_0 := \partial \Omega \setminus \Gamma$.

By assumption,  the entries $T^{j_1\cdots j_m}$ in the proposition vanish to infinite order on the boundary.
 For $m=1$, the integral identity \eqref{prop1equ} reads
\begin{align*}
0 & =\sum_{\left(l_1, l_2\right) \in \pi(2)} \sum_{j=0}^n \int T^j(x)\left(u_{l_1}, \nabla u_{l_1}\right)_j \nabla u_{l_2} \cdot \nabla u_3 \\
&=\sum_{j=0}^n \int T^j(x)\left(u_{1}, \nabla u_{1}\right)_j \nabla u_{2} \cdot \nabla u_3+\sum_{j=0}^n \int T^j(x)\left(u_{2}, \nabla u_{2}\right)_j \nabla u_{1} \cdot \nabla u_3,
\end{align*}

By setting $u_2=1$ in the integral identity, the terms involving $\nabla u_2$ all vanish and so   we have
\begin{equation}\label{s3e1}
    \int  T^0(x) \nabla u \cdot \nabla v  = 0 
\end{equation}
for any functions  $u$ and $v$ solving
\begin{equation}\label{cond_equ}
    \nabla \cdot\left(\gamma_0 \nabla u\right)=0 \text { in } \Omega .
\end{equation}
We rewrite \eqref{s3e1} as
$$
\int \left(\frac{T^0}{\gamma_0}\right) \gamma_0 \nabla u \cdot \nabla v = 0
$$
We integrate by parts to move the gradient on $v$ to the other terms. By using also \eqref{cond_equ}, we get
\begin{align}\label{ibps3e1}
     \int \gamma_0 v \nabla \left(\frac{T^0}{\gamma_0}\right)\cdot \nabla u =0.
\end{align}
There is no boundary term since $T^0$ vanishes on the boundary. 
%Again using \eqref{cond_equ}, integration by parts gives \f{More details here would be nice. There was also inner product sign missing from somewhere.}
% \begin{align*}
% \int \gamma_0 \left[v \nabla \left(\frac{T^0}{\gamma_0}\right)+(\nabla v) \frac{T^0}{\gamma_0}\right]\cdot \nabla u  & = \int \gamma_0 \nabla  \left[v  \left(\frac{T^0}{\gamma_0}\right)\right]\cdot \nabla u\\
% & = \int v  \left(\frac{T^0}{\gamma_0}\right) \nabla \cdot (\gamma_0 \nabla u)
% = 0.
% \end{align*}
%There is no boundary term since we are now assuming $T^{\s 0}\in C_c^\infty(\Omega)$. 
By integrating by parts again, we obtain 
\begin{align}\label{s3uvprod}
  0 =  \int \gamma_0 v \nabla \left(\frac{T^0}{\gamma_0}\right) \cdot \nabla u 
    = \int uv \nabla \cdot \left(\gamma_0 \nabla\left(\frac{T^0}{\gamma_0}\right)\right) + 
    \int \gamma_0 u \nabla \left(\frac{T^0}{\gamma_0}\right) \cdot \nabla v \nonumber \\
    = \int uv \nabla \cdot \left(\gamma_0 \nabla\left(\frac{T^0}{\gamma_0}\right)\right).
\end{align}
Here in the last identity we used \eqref{ibps3e1} with $u$ in place of $v$. 

Let us denote 
\[
A:=\nabla \cdot \left(\gamma_0 \nabla\left(\frac{T^0}{\gamma_0}\right)\right),
\]
and let $u$ and $v$ be the CGO solutions 
\begin{align*}
u &= \frac{1}{\sqrt{\gamma_0}} (e^{\Phi/h}(a + h a_0 + r_1)
      + \overline{e^{\Phi/h}(a + h a_0 + r_1)}
      + e^{\varphi/h} r_2) \\
v &= \frac{1}{\sqrt{\gamma_0}}(e^{-\Phi/h}(a + h b_0 + s_1)
      + \overline{e^{-\Phi/h}(a + h b_0 + s_1)}
      + e^{-\varphi/h} s_2)
\end{align*}
where $\Phi =\varphi +i\psi$ is a holomorphic function that is purely real on $\Gamma_0$ and Morse on $\Omega$, with a critical point at $z_0\in \Omega$. We have by proof of Proposition 5.1 in \cite{GuillarmouTzou2011} that
 $A(z_0)\gamma_0^{-1}(z_0)=0$. This shows that $A=0$ at $z=z_0$. Since point $z_0$ is chosen arbitrarily, we can vary the critical point of the phases of the CGOs to show that $A=0$ in $\Omega$, 
$$
\nabla \cdot \left(\gamma_0 \nabla\left(\frac{T^0}{\gamma_0}\right)\right) = 0\quad \text{in $\Omega$}.
$$
Thus $T^0/\gamma_0$ is a solution to an elliptic equation in $\Omega$.  
Since $T^0$ is identically zero on the boundary by assumption, uniqueness of solutions to the Dirichlet problem of the above equation shows that 
\[
T^0=0 \text{ in } \Omega.
\]

By using that $T^0\equiv 0$, the integral identity \eqref{prop1equ} we started from now reduces to 
\begin{equation}\label{eq3}
    \sum_{\left(l_1, l_2\right) \in \pi(2)} \sum_{j=1}^2 \int T^j(x) \partial_{x_j} u_{l_1} \nabla u_{l_2} \cdot \nabla u_3 d x = 0,
\end{equation}
in the current case $m=1$ holding  for all $u_l\in {C}^{\infty}(\overline{\Omega})$, $l=1,2,3$, solving \ref{cond_equ}.

Now let $\Theta_1 := \Psi + \Phi$, $\Theta_2 = -\Psi + \Phi$ and 
$\Theta_3 = -2\bar{\Phi}$, where $\Phi$ and $\Psi$ are holomorphic, $\Phi$ is Morse with a critical point at $z_0$,  $\Theta_1$ and $\Theta_2$ do not have critical points in $\Omega$. In addition, we ask $\Theta_i$ are all purely real on $\Gamma_0$, and $\Phi(z_0)\neq 0$. It is shown in \cite{GuillarmouTzou2011} that such $\Phi$ exists, so that $\Theta_3$ sastisfying the above requirement exists.

For $\Theta_1$ an $\Theta_2$,  this can be achieved by
choosing proper $\lambda\in \C$ and $f$ holomorphic,
\[
\Theta_1 = \lambda f, 
\qquad 
\Theta_2 = 2\Phi - \lambda f
% \qquad\text{and}\qquad
% \Theta_3 = -2\bar{\Phi},
\]
We first let $f$ be a conformal map that maps $\Gamma_0$ to part of the real line, so $f'\neq 0$ on $\Omega$, and $\Theta_1$ has no critical point on $\Omega$ and is real on $\Gamma_0$. Now since $f'\neq 0$ on $\Omega$, the ratio  
$$
r(z) = \frac{2\Phi'(z)}{f'(z)}
$$
is holomorphic on $\Omega$,
and its image is some bounded subset of $\C$. Let $\lambda$ be chosen outside the image of $r$, then $\Theta_2'=2\Phi' - \lambda f'$ is nonzero on $\Omega$. Also, $\Theta_2$ is real on $\Gamma_0$ since $f$ and $\Phi$ are. 

Let $a$ be holomorphic, $a(z_0) \neq 0$, and $a$ vanishes to high order at other possible critical points of $\Phi$. We continue by letting $u_l$, $l=1,2,3$, to be solutions as following 
\begin{align*}
u_1 &= \frac{1}{\sqrt{\gamma_0}}(e^{\Theta_1/h}(a + h a_0 + r_1)
      + \overline{e^{\Theta_1/h}(a + h a_0 + r_1)}
      + e^{\varphi_1/h} r_2) \\
u_2 &= \frac{1}{\sqrt{\gamma_0}}(e^{\Theta_2/h}(a + h b_0 + s_1)
      + \overline{e^{\Theta_2/h}(a + h b_0 + s_1)}
      + e^{\varphi_2/h} s_2)\\
u_3 &=\frac{1}{\sqrt{\gamma_0}}( e^{\Theta_3/h} \overline{({a} + h c_0 + t_1)}
      + \overline{e^{\Theta_3/h}}({a} + h c_0 + t_1)
      + e^{\varphi_3/h} t_2)
\end{align*}
Notice that we can rewrite these solutions as
\begin{align}\label{rewritesolu}
   {u_1}  =\frac{1}{\sqrt{\gamma_0}}v_1, \quad
   {u_2}  =\frac{1}{\sqrt{\gamma_0}}v_2, \quad
   {u_3}  =\frac{1}{\sqrt{\gamma_0}}v_3,
\end{align}
where the functions $v_l$ solve
$$
\Delta v +qv =0
$$
with $q = \Delta \gamma_0/\sqrt{\gamma_0}$.

Before we proceed, let us give some estimates for correction terms in the above notation by lemmas in section \ref{section:CGO}. By Lemma \ref{r_1: L2 est} and Lemma \ref{r_2: L2 est}, we have $||t_1||_{L^2} = O(h)$, $||t_2||_{L^2}=O(h^{\frac{3}{2}}\log h)$. Moreover, we have by Lemma \ref{r1: H^2 est original} and Lemma \ref{r_2:  est new} that $||t_1||_{H^2} = o(\frac{1}{h})$, $||t_2||_{H^1_{loc}} = O(h^{\frac{1}{2}}\log h)$.

Also, by Lemma \ref{r_1: H2 est new} we have for $r_1$ and $s_1$,
$$||r_1||_{H^2}, ||s_1||_{H^2} = O(h).$$

% \begin{align*}
%     r_1  &= r_{11}+hr_{12} \\
%     r_{11}:&= \chi e^{-2i\psi /h }\bar{\partial}^{-1}(e^{2i\psi/h}\chi_1 b)
% \end{align*}
% where $r_{12}$ is independent of $h$.
% Then since 
% $$
% \bar{\partial}^{-1} e^{\mathrm{i} \varphi/h} f=\frac{ih}{2}\left[e^{\mathrm{i} \varphi/h} \frac{f}{\bar{\partial} \varphi}+ \frac{1}{2}\bar{\partial}^{-1} \left(e^{\mathrm{i}  \varphi/h} \bar{\partial}\left(\frac{f}{\bar{\partial} \varphi}\right)\right)\right],
% $$
%  holds for any $f \in C_0^1(\bar{\Omega})$,
% $||r_1||_{H^1}, ||s_1||_{H^1} = O(h)$.
% We also have 

Moreover, by Lemma \ref{r_2:  est new}, we have for correction terms $r_2, s_2$ 
 $$||r_2||_{L^2}, ||s_2||_{L^2}=O(h^{\frac{3}{2}}),$$ and 
%Additionally, since $r_2$ ($s_2$) solves
% $$e^{\varphi/h}(\Delta + q)e^{-\varphi/h}v = \eta=O(h),$$
% for $
% \eta := e^{-2 i \psi / h}\, R\!\left(e^{2 i \psi / h}\,\chi_1 b\right)\,\partial \chi $, 
% by elliptic regularity we have
$$
||r_2||_{H^1_{loc}} = O(h^{\frac{1}{2}}), ||r_2||_{H^2_{loc}} = O(h^{-\frac{1}{2}}), 
$$
for any fixed $h>0$.

Now since
$$
\nabla u_{l_2} \cdot \nabla u_3 = \frac{1}{2}(\Delta (u_{l_2}u_3) - u_{l_2}\Delta u_3 - (\Delta u_{l_2})u_3),
$$
we can rewrite \eqref{eq3} as
$$
 \sum_{\left(l_1, l_2\right) \in \pi(2)} \sum_{j=1}^2 \int (T^j \partial_{x_j} u_{l_1}) (\Delta (u_{l_2}u_3) - u_{l_2}\Delta u_3 - (\Delta u_{l_2})u_3)=0
$$
By using \eqref{rewritesolu}, we also have
\begin{align}\label{Deltaid}
\Delta u_l = \Delta\left(\frac{1}{\sqrt{\gamma_0}}\right)v_l + 2\nabla\left(\frac{1}{\sqrt{\gamma_0}}\right)\cdot \nabla v_l - \frac{qv_l}{\sqrt{\gamma_0}}.
\end{align}
Let us write $\Delta=4\partial \bar{\partial}$ and integrate by parts to obtain 
\begin{align}\label{eq:second integral}
  0 &= \sum_{\left(l_1, l_2\right) \in \pi(2)} \sum_{j=1}^2 \int (T^j \partial_{x_j} u_{l_1}) (\Delta (u_{l_2}u_3) - u_{l_2}\Delta u_3 - (\Delta u_{l_2})u_3) \nonumber \\ 
  &= \sum_{\left(l_1, l_2\right) \in \pi(2)} \sum_{j=1}^2 \int \Bigg\{ 4\partial\bar{\partial}(T^j \partial_{x_j} u_{l_1}) (u_{l_2}u_3) 
  \\ 
  & \qquad \qquad-(T^j \partial_{x_j} u_{l_1})(2u_{l_2}\nabla(\frac{1}{\sqrt{\gamma_0}})\cdot \nabla v_3 + 2u_{3}\nabla(\frac{1}{\sqrt{\gamma_0}})\cdot \nabla v_{l_2}) \nonumber \\ 
   &\qquad\qquad\qquad\qquad\qquad\qquad\qquad\qquad\qquad+  (T^j \partial_{x_j} u_{l_1})F_{T,\gamma_0}(u_{l_2}v_3+v_{l_2}u_3)\Bigg\}, \nonumber
\end{align}
where $F_{T,\gamma_0}$ is a general term that depends smoothly only on $T^l$ and $\gamma_0$, and their derivatives, and later $F_{T,\gamma_0}$ may vary from line to line.

Now let $A=a+ha_0+r_1,B=a+hb_0+s_1, {C}={a+hc_0+t_1}$, we first compute
\begin{align}\label{v1v2v3}
    v_1v_2v_3 & =  \nonumber 2\Re(e^{\frac{2(\Phi - \bar{\Phi})}{h}}AB\bar{C}+ AB{C})+ 2\Re (e^{\frac{
\lambda f - \overline{{\lambda f}} }{h}}A\bar{B}\bar{C} + e^{\frac{{
\lambda f - \overline{\lambda f}} - 2{\Phi}+2\overline{\Phi} }{h}}A\bar{B}{C}) \\ \nonumber
    & + (e^{\varphi_1/h} r_2) \Re(e^{(\Theta_2+\Theta_3)/h}B\bar{C}+e^{(\Theta_2+\overline\Theta_3)/h}B{C})+ (e^{\varphi_2/h} s_2) \Re(e^{(\Theta_1+\Theta_3)/h}A\bar{C}+e^{(\Theta_1+\overline\Theta_3)/h}A{C})\\& \nonumber + (e^{\varphi_3/h} t_2) \Re(e^{(\Theta_1+\Theta_2)/h}A{B}+e^{(\Theta_1+\overline\Theta_2)/h}A\bar{B}) \\ 
    & + e^{\varphi_1 + \varphi_2 + \varphi_3} r_2 s_2 t_2 
\end{align}
and 
\[
\begin{aligned}
AB\overline C
&=a^2\bar a
+a^2\overline{t_1}
+a\bar a(r_1+s_1)
+h\Big(
a^2\overline{c_0}
+a\bar a(a_0+b_0)\Big)\\
&\quad + h\Big(a(a_0+b_0)\overline{t_1}
+(a_0s_1+b_0r_1)\bar a
+(a_0s_1+b_0r_1)\overline{t_1}
+a(r_1+s_1)\overline{c_0}
+r_1s_1\overline{c_0}
\Big) \\[2mm]
&\quad
+h^2\Big(
a_0b_0\bar a
+a_0b_0\overline{t_1}
+a(a_0+b_0)\overline{c_0}
+(a_0s_1+b_0r_1)\overline{c_0}
\Big) \\[2mm]
&\quad
+h^3 a_0b_0\overline{c_0}+a(r_1+s_1)\overline{t_1}
+r_1s_1\bar a
+r_1s_1\overline{t_1}\\
&\quad = a^2\bar a
+a^2\overline{t_1}
+a\bar a(r_1+s_1)
+h\Big(
a^2\overline{c_0}
+a\bar a(a_0+b_0)\Big) +o(h)
\end{aligned}
\]
Similarly,
\begin{align*}
ABC
= & a^3  + \big[\,h\,a^2(a_0 + b_0 + c_0) + a^2(r_1 + s_1 + t_1)\big]+ o(h)
\end{align*}
where we used the estimate on the correction terms, as in Lemma \ref{r_1: L2 est}.
Moreover,
% \begin{align*}
% (\partial A)BC
% & = a^2\partial a + h[a(\partial a)(b_0+c_0)
% +a^2\partial a_0]+a(\partial a)(s_1+t_1)
% +a^2\partial r_1+o(h)
% \end{align*}
\begin{align*}
(\partial A)BC
& = a^2\partial a + h[a(\partial a)(b_0+c_0)
+a^2\partial a_0]+a(\partial a)(s_1+t_1)
+a^2\partial r_1+o(h)
\end{align*}

We first look at the term $ \sum_{\left(l_1, l_2\right) \in \pi(2)}\int (T^j \partial_{x_j} u_{l_1}F_{T,\gamma_0})(u_{l_2}v_3+v_{l_2}u_3)$.
Fix $l_1 = 1$ first; the case $ l_1 = 2$ is similar.
Now we have
$$
\partial_j u_{1} = \frac{1}{\sqrt{\gamma_0}} \partial_j v_1
+ \partial_j(\frac{1}{\sqrt{\gamma_0}}) v_1$$
and

\begin{align*}
\partial_j v_{1}
={}&
\frac{1}{h}(\partial_j \Theta_{1})\,e^{\Theta_{1}/h}
A
+ e^{\Theta_{1}/h}\,\partial_jA \\
&+
\frac{1}{h}(\partial_j \overline{\Theta_{1}})\,
e^{\overline{\Theta_{1}}/h}\,
\overline{A}
+ e^{\overline{\Theta_{1}}/h}\,
\overline{\partial_j A} \\
&+
\frac{1}{h}(\partial_j \varphi_{1})\,e^{\varphi_{1}/h}r_2
+ e^{\varphi_{1}/h}\,\partial_j r_2.
\end{align*}

For the second term $\partial_j (\frac{1}{\sqrt{\gamma_0}})v_1$ in  $\partial_j u_1$, the corresponding integral is 
$$
\int T^j F_{T,\gamma_0} v_1 v_2v_3 = O(1).
$$
For the first term in $\partial_j u_1$, we observe that we only need to focus on 
$\frac{1}{h}(\partial_j \Theta_{1})\,e^{\Theta_{1}/h}
A$ and $\frac{1}{h}(\partial_j \overline{\Theta_{1}})\,
e^{\overline{\Theta_{1}}/h}\,
\overline{A}$ in $\partial_j v_1$, as other terms are also $O(1)$, where we used Lemma \ref{r_2:  est new} for estimating $r_2$. The integral corresponding to these two terms would be 
$$
\int \frac{1}{h} T^j F_{T,\gamma_0} (\partial_j \Theta_1) \left((e^{\frac{2(\Phi - \bar{\Phi})}{h}}AB\bar{C}+ AB{C})+ 2(e^{\frac{
\lambda f - \overline{{\lambda f}} }{h}}A\bar{B}\bar{C} + e^{\frac{{
\lambda f - \overline{\lambda f}} - 2{\Phi}+2\overline{\Phi} }{h}}A\bar{B}{C})\right) +O(1)
$$
and 
$$
\int \frac{1}{h}T^j F_{T,\gamma_0} (\partial_j \overline{\Theta_1}) \overline{\left((e^{\frac{2(\Phi - \bar{\Phi})}{h}}AB\bar{C}+ AB{C})+ (e^{\frac{
\lambda f - \overline{{\lambda f}} }{h}}A\bar{B}\bar{C} + e^{\frac{{
\lambda f - \overline{\lambda f}} - 2{\Phi}+2\overline{\Phi} }{h}}A\bar{B}{C})\right)} +O(1)
$$
respectively, where we used remainder estimate on $s_2$ and $t_2$. 
Now since $\lambda f$ and $\lambda f - 2\Phi$ has no critical point, and the $H^1$ norm of $A$, $B$, $C$ are $O(1)$, we have by non stationary phase that 
$$
\int \frac{1}{h} (\partial_j \Theta_1) \left( e^{\frac{
\lambda f - \overline{{\lambda f}} }{h}}A\bar{B}\bar{C} + e^{\frac{{
\lambda f - \overline{\lambda f}} - 2{\Phi}+2\overline{\Phi} }{h}}A\bar{B}{C}\right) = O(1)
$$
and 
$$
\int \frac{1}{h} (\partial_j \overline{\Theta_1}) \overline{\left(e^{\frac{
\lambda f - \overline{{\lambda f}} }{h}}A\bar{B}\bar{C} + e^{\frac{{
\lambda f - \overline{\lambda f}} - 2{\Phi}+2\overline{\Phi} }{h}}A\bar{B}{C}\right)} =O(1)
$$
Also, since $\Phi$ has a critical point at $z_0$, by stationary phase, 
$$
\int \frac{1}{h} (\partial_j \Theta_1) \left(e^{\frac{2(\Phi - \bar{\Phi})}{h}}AB\bar{C}\right) =O(1),
$$
and 
$$
\int \frac{1}{h} (\partial_j \overline{\Theta_1}) \overline{\left(e^{\frac{2(\Phi - \bar{\Phi})}{h}}AB\bar{C}\right)} =O(1).
$$
If $l_1 = 2$ and $l_2 = 1$, the difference in the argument is to replace $\Theta_1$ with $\Theta_2$. 

Hence, we have shown that 
\begin{align*}
 \sum_{\left(l_1, l_2\right) \in \pi(2)}\int (T^j F_{T,\gamma_0} \partial_{x_j} u_{l_1})(u_{l_2}v_3+v_{l_2}u_3)= \frac{1}{h} \int T^j F_{T,\gamma_0} 2\Re \left( (\partial_j \Theta_1+\partial_j \Theta_2)AB{C}  
\right)+O(1).
\end{align*}
In addition, we compute
\begin{align}\label{eq_m1_first_result}
\frac{1}{h} \int T^j F_{T,\gamma_0} 2\Re \left( (\partial_j \Theta_1+\partial_j \Theta_2)AB{C}  
\right) = \frac{1}{h} \int T^j F_{T,\gamma_0} 2\Re((\partial_j \Theta_1+\partial_j \Theta_2)a^3)+ O(1).
\end{align}
by expanding $ABC$, and using Lemma \ref{r_1: L2 est}.

% If $\partial_{x_j}$ acts on $e^{\Theta_1/h}$, the result is
% $$
% \int \frac{1}{h} {F}_{T,\gamma_0} (\partial_j \Theta_1 )(AB\bar{C}+\overline{AB}C)+o(\frac{1}{h})
% $$
% Note we used stationary phase on the term including $e^{\frac{2(\Phi - \bar{\Phi})}{h}}ABC$, and non stationary phase (and $\partial(ABC)=O(1)$) on the other term $2\Re (e^{\frac{\Psi-\Phi - \bar{\Psi}+\bar{\Phi}}{h}}A\bar{B}C + e^{\frac{\Psi+\Phi - \bar{\Psi}-\bar{\Phi}}{h}}A\bar{B}\bar{C})$ in the expression we computed for $v_1 v_2 v_3$. We also used estimates on remainders for the other terms in $v_1 v_2 v_3$ with $r_2$, $s_2$ and $t_2$.

% If $\partial_{x_j}$ acts on $\frac{1}{\sqrt{\gamma_0}}$, the result is $O(1)$; if it acts on $A$ or $e^{\varphi_1/h}r_2$, the result is also $O(1)$.

% we have
% \begin{equation}\label{eq:Tgamma_id}
% %\sum_{\left(l_1, l_2\right) \in \pi(2)} \sum_{j=1}^2 
% \int (T^j \partial_{x_j} u_{l_1}F_{T,\gamma_0})(u_{l_2}v_3+v_{l_2}u_3)=O(1).
% \end{equation}

Using the above and by writing   $\nabla u \cdot \nabla v = 2(\partial u \bar{\partial}v+ \bar{\partial} u \partial v)$, the integral in  \eqref{eq:second integral} reads   
\begin{align}\label{longrhs}
%0& =  \sum_{\left(l_1, l_2\right) \in \pi(2)} \sum_{j=1}^2 
&\int 4\partial\bar{\partial}(T^j \partial_{x_j} u_{l_1}) (u_{l_2}u_3) - \int 4(T^j \partial_{x_j} u_{l_1}) \left[u_{l_2}\left(\partial \left(\frac{1}{\sqrt{\gamma_0}}\right)  \bar{\partial}v_3+\bar{\partial} \left(\frac{1}{\sqrt{\gamma_0}}\right)  {\partial}v_3\right)\right. \nonumber \\ \nonumber
&\qquad\quad\qquad\qquad\qquad\qquad\qquad\qquad \left.+u_3\left(\partial \left(\frac{1}{\sqrt{\gamma_0}}\right)  \bar{\partial}v_{l_2} +\bar{\partial} \left(\frac{1}{\sqrt{\gamma_0}}\right) {\partial}v_{l_2}\right)\right] \\ & \qquad\quad\qquad\qquad\qquad\qquad\qquad\qquad +\frac{1}{h} \int T^j F_{T,\gamma_0} 2\Re((\partial_j \Theta_1+\partial_j \Theta_2)a^3) +O(1).
\end{align}
Here we have 
\begin{multline}\label{first_part}
    \int 4\partial\bar{\partial}(T^j \partial_{x_j} u_{l_1}) (u_{l_2}u_3)=\int 4(\partial T^j \bar{\partial} \partial_{x_j} u_{l_1} +  \bar{\partial} T^j {\partial} \partial_{x_j} u_{l_1}) (u_{l_2}u_3) \\
+\int4T^j\partial_{x_j}\left(\partial \left(\frac{1}{\sqrt{\gamma_0}}\right) \bar{\partial}v_{l_1}+ \bar{ \partial} \left(\frac{1}{\sqrt{\gamma_0}}\right) {\partial}v_{l_1}\right) (u_{l_2}u_3) +\frac{1}{h} \int T^j F_{T,\gamma_0} 2\Re((\partial_j \Theta_1+\partial_j \Theta_2)a^3) +O(1)
\end{multline}
since in the case both 
%since for the first term on the right hand side, if 
$\partial$ and $\bar{\partial}$ hit $T^j$ in the term $4\partial\bar{\partial}(T^j \partial_{j} u_{l_1})$,  then it reduces to the previous case, and we get a term that can be absorbed into the last terms. We also used $\Delta v+qv=0$ for the second term on the right.

Now, we look at
\[
 \int 4 (\partial T^j \bar{\partial} \partial_{j} u_{l_1})(u_{l_2}u_3).
\]
Recall the notations  $\partial_1 = \partial+\bar{\partial}$ and $ \partial_2= i(\partial-\bar{\partial})$. So we have
\begin{equation}\label{pp12}
\partial \partial_{1} = \partial^2+\partial \bar{\partial}, \quad
\partial \partial_{2} = i(\partial^2-\partial \bar{\partial})
\end{equation}
and
\begin{equation}\label{barpp12}
\bar\partial\,\partial_1=\bar\partial\,\partial+\bar\partial^2,\qquad
\bar\partial\,\partial_2=i\bigl(\bar\partial\,\partial-\bar\partial^2\bigr).
\end{equation}
Now notice
\[
u_{l_1}=\gamma_0^{-1/2}v_{l_1},
\]

\[
\partial\bar\partial u_{l_1}
=(\partial\bar\partial \gamma_0^{-1/2})\,v_{l_1}
+(\bar\partial \gamma_0^{-1/2})\,\partial v_{l_1}
+(\partial \gamma_0^{-1/2})\,\bar\partial v_{l_1}
-\frac{q}{4}\gamma_0^{-1/2}\,v_{l_1},
\]
so the first and last term are $O(1)$, and two middle terms will reduce to the previous case where only order one derivative hits $v_{l_1}.$
Hence by \eqref{barpp12}, we have
\begin{align}\label{eq_m1_first_principal}
\sum_{\left(l_1, l_2\right) \in \pi(2)} \sum_{j=1}^2 \int  4 {\partial} T^j ({\bar{\partial} \partial_{j}} u_{l_1}) u_{l_2}u_3=&\sum_{\left(l_1, l_2\right) \in \pi(2)}  \int  4 {\partial} (T^1-iT^2) (\bar{\partial}^2 u_{l_1}) (u_{l_2}u_3) \nonumber \\& +\int \frac{1}{h}  \left( {F}_{T,\gamma_0}(\partial \Theta_1+\partial \Theta_2)a^3+\tilde{F}_{T,\gamma_0}(\bar{\partial} \overline{\Theta_1}+\bar{\partial} \overline{\Theta_2}))\overline{a}^3\right) +O(1)
\end{align}
Now we consider
\[
\bar{\partial}^2 u_{l_1}
=(\bar{\partial}^2\gamma_0^{-1/2})\,v_{l_1}
+2(\bar{\partial}\gamma_0^{-1/2})\,\bar{\partial}v_{l_1}
+\gamma_0^{-1/2}\,\bar{\partial}^2 v_{l_1}.
\]
 The last term $\gamma_0^{-1/2}\,\bar{\partial}^2 v_{l_1}$ would give largest terms since it contains second derivative of $v_{l_1}$.
 The first term $(\bar{\partial}^2\gamma_0^{-1/2})\,v_{l_1}$ would give an $O(1)$ term.
 
Substituting the middle term $2(\bar{\partial}\gamma_0^{-1/2})\,\bar{\partial}v_{l_1}$ into the integral in \eqref{eq_m1_first_principal}
$$
\sum_{\left(l_1, l_2\right) \in \pi(2)}  \int  4 {\partial} (T^1-iT^2) (\bar{\partial}^2 u_{l_1}) (u_{l_2}u_3),
$$
we would get one term of order $\frac{1}{h}$:
$$
\int \frac{1}{h} F_{T,\gamma_0} (\bar{\partial} \overline{\Theta_1}+\bar{\partial} \overline{\Theta_2} )\overline{a}^3,
$$
and other terms of lower order, using a similar argument as showing \eqref{eq_m1_first_result}.

For the last term $\gamma_0^{-1/2}\,\bar{\partial}^2 v_{l_1}$, we first consider the case $l_1=1$, while the case $l_1=2$ would be similar.
Recall 
$$
v_1 =e^{\Theta_1/h}A
      + \overline{e^{\Theta_1/h}A}
      + e^{\varphi_1/h} r_2
$$
and so
\begin{align}\label{second_deriv_v}
\bar\partial^2 v_{1}
&=\frac{1}{h^2}\Big(
e^{\overline{\Theta_1}/h}\,(\bar\partial \overline{\Theta_1})^{2}\,\overline{A}
\;+\;
e^{\varphi_1/h}\,(\bar\partial \varphi_1)^{2}\,r_2
\Big) \nonumber \\[1mm] \nonumber
&\quad
+\frac{1}{h}\Big(
e^{\overline{\Theta_1}/h}\big[
(\bar\partial^2 \overline{\Theta_1})\,\overline{A}
+2(\bar\partial \overline{\Theta_1})\,\bar\partial \overline{A}
\big]
\;+\;
e^{\varphi_1/h}\big[
(\bar\partial^2 \varphi_1)\,r_2
+2(\bar\partial \varphi_1)\,\bar\partial r_2
\big]
\Big)\\[1mm] \nonumber
&\quad
+\Big(
e^{\Theta_1/h}\,\bar\partial^{2}A
+
e^{\overline{\Theta_1}/h}\,\bar\partial^{2}\overline{A}
+
e^{\varphi_1/h}\,\bar\partial^{2}r_2
\Big)\\[1mm]
\quad  &=
\frac{1}{h^2}\Big(
e^{\overline{\Theta_1}/h}\,(\bar\partial \overline{\Theta_1})^{2}\,\overline{A}
\;\Big)+\frac{1}{h}\Big(
e^{\overline{\Theta_1}/h}\big[
(\bar\partial^2 \overline{\Theta_1})\,\overline{A}
+2(\bar\partial \overline{\Theta_1})\,\bar\partial \overline{A}
\big]
\;\Big)+o\left(\frac{1}{h}\right)
\end{align}
where we used remainder estimates on $r_1$, $r_2$.

Substituting the second term on the right hand side of \eqref{second_deriv_v} to the integral in \eqref{eq_m1_first_principal}, by a similar computation as in showing \eqref{eq_m1_first_result}, the corresponding integral would be 
\begin{align}\label{quant_m1_2}
& \nonumber \int \frac{1}{h} F_{T,\gamma_0} \big[
(\bar\partial^2 \overline{\Theta_1})\,\overline{ABC}
+2(\bar\partial \overline{\Theta_1})\,(\bar\partial \overline{A})\overline{{BC}}
\big] +O(1) \\
& = \frac{1}{h}\int F_{T,\gamma_0}[\bar\partial^2 \overline{\Theta_1})\,\overline {a}^3+ 2(\bar\partial \overline{\Theta_1})\,\bar a^2(\bar\partial \bar a ) ]+ O(1)
\end{align}
For the first term on the right of \eqref{second_deriv_v}, the corresponding integral is 
\begin{align}\label{quant_m1_3}
     &\int \frac{4}{h^2\gamma_0^{3/2}} \partial(T^1-iT^2)(\bar\partial \overline{\Theta_1})^{2} e^{\overline{\Theta}_1/h} \bar{A} v_2v_3
 \nonumber  \\ 
 \nonumber & = \nonumber \frac{4}{h^2} \int F_{T,\gamma_0} (\bar\partial \overline{\Theta_1})^{2} \overline{ABC}
 + \frac{4}{h^2} \int F_{T,\gamma_0} (\bar\partial \overline{\Theta_1})^{2} \overline{ (e^{\frac{
\lambda f - \overline{{\lambda f}} }{h}}A\bar{B}\bar{C} + e^{\frac{{
\lambda f - \overline{\lambda f}} - 2{\Phi}+2\overline{\Phi} }{h}}A\bar{B}{C})} 
 \\& \nonumber  +  \frac{4}{h} C_{z_0} \gamma_0^{-3/2} \left(\bar{a}|a|^2\right) {\partial} (T^1-iT^2) (\bar\partial \overline{\Theta_1})^{2}  e^{(\bar\Phi-{\Phi})/h}|_{z=z_0}+ O(1)\\
 & = \nonumber \frac{4}{h^2} \int F_{T, \gamma_0} (\bar\partial \overline{\Theta_1})^{2} \bar{a}^3  \\& \nonumber + \frac{4}{h^2} \int F_{T, \gamma_0} (\bar\partial \overline{\Theta_1})^{2} \bar{a}^2 (\bar{r}_1+\bar{s}_1+\bar{t}_1) + \frac{4}{h} \int F_{T, \gamma_0} (\bar\partial \overline{\Theta_1})^{2}\bar{a}^2(\bar{a}_0+\bar{b}_0+\bar{c}_0) \\
 & + \frac{4}{h} C_{z_0} \gamma_0^{-3/2} \left(\bar{a}|a|^2\right) {\partial} (T^1-iT^2) (\bar\partial \overline{\Theta_1})^{2}  e^{(\bar\Phi-{\Phi})/h}|_{z=z_0}+ o(\frac{1}{h})
\end{align}
where we used integration part for the second nonstationary phase integral and the estimate $||A||_{H^2_{loc}}, ||B||_{H^2_{loc}}, ||C||_{H^2_{loc}}=o(\frac{1}{h})$.

We remark here that the above computation shows the general pattern for all the later computations for the integral of a product of (derivatives) of solutions. After expansion, the right-hand side can be divided into three types:  integral of highest order in $h$ with no exponential term, or of one lower order with an exponential term with a phase that has no critical point, or of one lower order with an exponential term with a phase that has a nondegenerate critical point. In the last case, we have used stationary phase method above and obtained an evaluation term at the critical point.
We will omit some details in the computation later that follows this pattern.

For $l_1=2$, we get same terms as in \eqref{quant_m1_2} and \eqref{quant_m1_3} except for replacing $\Theta_1$ by $\Theta_2$.

Similarly, we look at 
$$
 \int 4\bar{\partial} T^j {\partial} \partial_{x_j} u_{l_1} (u_{l_2}u_3)
$$
back in \eqref{first_part}.
Then by \eqref{pp12}, we have
\begin{align}\label{eq_m1_second_principal}
\sum_{\left(l_1, l_2\right) \in \pi(2)} \sum_{j=1}^2 \int  4 \bar{\partial} T^j ({{\partial} \partial_{j}} u_{l_1}) u_{l_2}u_3=&\sum_{\left(l_1, l_2\right) \in \pi(2)}  \int  4 \bar{\partial} (T^1+iT^2) ({\partial}^2 u_{l_1}) (u_{l_2}u_3) \nonumber \\& +\int \frac{1}{h}  \left( {F}_{T,\gamma_0}(\partial \Theta_1+\partial \Theta_2)a^3+\tilde{F}_{T,\gamma_0}(\bar{\partial} \overline{\Theta_1}+\bar{\partial} \overline{\Theta_2}))\overline{a}^3\right) +O(1)
\end{align}
And we have quantities similarly to \eqref{quant_m1_2} and \eqref{quant_m1_3}:
\begin{align}\label{quant_m1_4}
\int \frac{1}{h} F_{T,\gamma_0} \big[
(\partial^2 {\Theta_1})\,{a}^3
+2(\partial {\Theta_1})\,(\partial {a}){{a}^2}
\big] +O(1)
\end{align}
and
\begin{align}\label{quant_m1_5}
     &\int \frac{4}{h^2\gamma_0^{3/2}} \bar{\partial}(T^1+iT^2)(\partial {\Theta_1})^{2} e^{{\Theta}_1/h} {A} v_2v_3
 \nonumber   \\ & 
=  \nonumber \frac{4}{h^2} \int F_{T, \gamma_0} (\partial {\Theta_1})^{2} {a}^3  \\& \nonumber + \frac{4}{h^2} \int F_{T, \gamma_0} (\partial {\Theta_1})^{2} {a}^2 ({r}_1+{s}_1+{t}_1) + \frac{4}{h} \int F_{T, \gamma_0} (\partial {\Theta_1})^{2}{a}^2({a}_0+{b}_0+{c}_0) \\
 & +  \frac{4}{h} C_{z_0} \gamma_0^{-3/2} \left({a}|a|^2\right) \bar{\partial} (T^1+iT^2)  (\partial {\Theta_1})^{2}  e^{(\Phi-\bar{\Phi})/h}|_{z=z_0}+ o(\frac{1}{h})
\end{align}

Next, we look at the second part in \eqref{first_part}:
$$
\int4T^j\partial_{j}\left(\partial \left(\frac{1}{\sqrt{\gamma_0}}\right) \bar{\partial}v_{l_1}+ \bar{ \partial} \left(\frac{1}{\sqrt{\gamma_0}}\right) {\partial}v_{l_1}\right) (u_{l_2}u_3)
$$
First consider
$$
\int4T^j\partial_{j}\left(\partial \left(\frac{1}{\sqrt{\gamma_0}}\right) \bar{\partial}v_{l_1}\right) (u_{l_2}u_3)
$$
If $\partial_{j}$ hits $\partial \left( \frac{1}{\sqrt{\gamma_0}}\right)$, then we reduce to the case where only order one derivative hits $v_{l_1}$, and we would get a same term as the in the second integral on the right hand side of \eqref{eq_m1_first_principal}, with a different coefficient $F_{T,\gamma_0}$. If $\partial_j$ hits $\bar{\partial}v_{l_1}$, we will use again \eqref{barpp12} and \eqref{second_deriv_v}(Note that $4\bar \partial \partial v = qv$). Hence we have (first consider $l_1=1$)
\begin{align}\label{quant_m1_6}
\nonumber &\int4T^j\left(\partial \left(\frac{1}{\sqrt{\gamma_0}}\right) \bar{\partial}\partial_{j}v_{l_1}\right) (u_{l_2}u_3) = \int 4(T^1-iT^2)\partial \left(\frac{1}{\sqrt{\gamma_0}}\right) (\bar{\partial}^2v_{l_1})u_{l_2}u_3 + O(1) \\
& = \nonumber \frac{4}{h^2} \int F_{T,\gamma_0} (\bar\partial \overline{\Theta_1})^{2} \overline{a}^3  \\& \nonumber  +  \frac{4}{h} C_{z_0} (\gamma_0^{-1} \left(\bar{a}|a|^2\right)\partial \left(\frac{1}{\sqrt{\gamma_0}}\right)  (T^1-iT^2) (\bar\partial \overline{\Theta_1})^{2}  e^{(\bar\Phi-{\Phi})/h}|_{z=z_0} \\& 
+\int \frac{1}{h} F_{T,\gamma_0} \big[
(\bar\partial^2 \overline{\Theta_1})\,\overline{a}^3
+2(\bar\partial \overline{\Theta_1})\,(\bar\partial \overline{a})\overline{{a}}^2
\big]+  o(\frac{1}{h})
\end{align}
For $l_1=2$, we get same terms as above except for replacing $\Theta_1$ by $\Theta_2$.

Similarly, consider 
$$
\int4T^j\partial_{j}\left(\bar \partial \left(\frac{1}{\sqrt{\gamma_0}}\right) {\partial}v_{l_1}\right) (u_{l_2}u_3)
$$
and we focus on (first for $l_1=1$)
\begin{align}\label{quant_m1_7}
\nonumber &\int4T^j\left(\bar\partial \left(\frac{1}{\sqrt{\gamma_0}}\right) {\partial}\partial_{j}v_{l_1}\right) (u_{l_2}u_3) = \int 4(T^1+iT^2)\bar\partial \left(\frac{1}{\sqrt{\gamma_0}}\right) ({\partial}^2v_{l_1})u_{l_2}u_3 + O(1) \\
&= \nonumber \frac{4}{h^2} \int F_{T,\gamma_0} (\partial {\Theta_1})^{2} {a}^3\\& \nonumber  +  \frac{4}{h} C_{z_0} (\gamma_0^{-1} \left({a}|a|^2\right)\bar\partial \left(\frac{1}{\sqrt{\gamma_0}}\right)  (T^1+iT^2) (\partial {\Theta_1})^{2}  e^{(\Phi-\bar{\Phi})/h}|_{z=z_0}   \\ & +\int \frac{1}{h} F_{T,\gamma_0} \big[
(\partial^2 {\Theta_1})\,{a}^3
+2(\partial {\Theta_1})\,(\partial {a}){{a}}^2
\big]z+  o(\frac{1}{h})
\end{align}
Similarly, we get the result for $l_1=2$ by replacing $\Theta_1$ by $\Theta_2$.

Next, we return to \eqref{longrhs} and consider the remaining integral  
\begin{align*}
   & \sum_{\left(l_1, l_2\right) \in \pi(2)} \sum_{j=1}^2 \int - 4T^j \partial_{x_j} u_{l_1} \Bigg\{ u_{l_2} \left[ \partial \left(\frac{1}{\sqrt{\gamma_0}}\right) \bar{\partial}v_3 +\bar{\partial}\left(\frac{1}{\sqrt{\gamma_0}}\right) {\partial}v_3\right] \\
&\qquad\quad\qquad\qquad\qquad\quad\qquad\qquad+u_3\left[ \partial \left(\frac{1}{\sqrt{\gamma_0}}\right)\bar{\partial}v_{l_2} +\bar{\partial} \left(\frac{1}{\sqrt{\gamma_0}}\right){\partial}v_{l_2}\right]\Bigg\}
\end{align*}

First, we consider $l_1=1$. Now note that since $\Theta_3$ is antiholomorphic, we have
\[
\partial v_3
=
e^{\overline{\Theta_3}/h}
\left(
\frac{1}{h}(\partial \overline{\Theta_3})C
+\partial C
\right)
+
e^{\varphi_3/h}
\left(
\frac{1}{h}(\partial \varphi_3)t_2
+\partial t_2
\right),
\]

\[
\bar\partial v_3
=
e^{\Theta_3/h}
\left(
\frac{1}{h}(\bar\partial \Theta_3)\overline{C}
+\bar\partial \overline{C}
\right)
+
e^{\varphi_3/h}
\left(
\frac{1}{h}(\bar\partial \varphi_3)t_2
+\bar\partial t_2
\right).
\]

And we recall
$$
\partial_j u_{1} = \frac{1}{\sqrt{\gamma_0}} \partial_j v_1
+ \partial_j(\frac{1}{\sqrt{\gamma_0}}) v_1$$
and

\begin{align*}
\partial_j v_{1}
={}&
\frac{1}{h}(\partial_j \Theta_{1})\,e^{\Theta_{1}/h}
A
+ e^{\Theta_{1}/h}\,\partial_jA \\
&+
\frac{1}{h}(\partial_j \overline{\Theta_{1}})\,
e^{\overline{\Theta_{1}}/h}\,
\overline{A}
+ e^{\overline{\Theta_{1}}/h}\,
\overline{\partial_j A} \\
&+
\frac{1}{h}(\partial_j \varphi_{1})\,e^{\varphi_{1}/h}r_2
+ e^{\varphi_{1}/h}\,\partial_j r_2.
\end{align*}

In $\partial v_3$, except the first term $\frac{1}{h}e^{\overline{\Theta_3}/h}
(\partial \overline{\Theta_3})C$, all other terms multiplying by $\partial_ju_{l_1}$ would reduce to cases in which only an order-one derivative hits $v_{l_1}$, and we obtain terms similar to those in \eqref{eq_m1_first_result}. For this first term mentioned, we compute similarly as in \eqref{quant_m1_3}, where we have an order $h^{-2}$ non oscillating  (without exponential) integral containing  $\partial \overline {\Theta}_3$, $\partial_j \Theta_1$,
\begin{align*}
\frac{1}{h^2}\int F_{T,\gamma_0} (\partial_j \Theta_1)(\partial\overline{\Theta}_3) ABC & = \frac{1}{h^2}\int F_{T,\gamma_0}(\partial_j \Theta_1)(\partial\overline{\Theta}_3)a^3 \\& + \frac{1}{h^2}\int F_{T,\gamma_0}(\partial_j \Theta_1)(\partial\overline{\Theta}_3)a^2(r_1+s_1+t_1) \\ &+ \frac{1}{h}\int F_{T,\gamma_0}(\partial_j \Theta_1)(\partial\overline{\Theta}_3)a^2(a_0+b_0+c_0) ,
\end{align*}
while we don't have the evaluation term which comes from using stationary phase, since $\bar{\partial} \Theta_3 = -2\bar{\partial}\bar{\Phi}=0$ at $z_0$. Other terms are $o(\frac{1}{h})$ as in \eqref{quant_m1_3}.

The situation is similar for $l_1=2$, and for $\bar{\partial}v_3$. 

Now, we turn to the last term 
\begin{align*}
   & \sum_{\left(l_1, l_2\right) \in \pi(2)} \sum_{j=1}^2 \int - 4T^j \partial_{x_j} u_{l_1} u_3\left[ \partial \left(\frac{1}{\sqrt{\gamma_0}}\right)\bar{\partial}v_{l_2} +\bar{\partial} \left(\frac{1}{\sqrt{\gamma_0}}\right){\partial}v_{l_2}\right]
\end{align*}
Again, we first fix $l_1=1$. Now we have 
% Assumptions: \Theta_2 holomorphic (\bar\partial\Theta_2=0), B holomorphic (\bar\partial B=0).

\[
\partial v_2
=
e^{\Theta_2/h}\left(\frac{1}{h}(\partial\Theta_2)\,B+\partial B\right)
+
e^{\varphi_2/h}\left(\frac{1}{h}(\partial\varphi_2)\,s_2+\partial s_2\right),
\]

\[
\bar\partial v_2
=
e^{\overline{\Theta_2}/h}\left(\frac{1}{h}(\bar\partial\overline{\Theta_2})\,\overline{B}+\bar\partial\overline{B}\right)
+
e^{\varphi_2/h}\left(\frac{1}{h}(\bar\partial\varphi_2)\,s_2+\bar\partial s_2\right).
\]

Thus, we can first compute 
\begin{align*}
\nonumber &\int4T^j\left(\partial \frac{1}{\sqrt{\gamma_0}}\right) \partial_{j}u_{l_1} \bar \partial v_{l_2}u_3 \\& \nonumber = \int 4(T^1+iT^2)\partial \left(\frac{1}{\sqrt{\gamma_0}}\right) ({\partial}u_{l_1})\bar \partial v_{l_2}u_3 + \int 4 (T^1-iT^2) \partial \left(\frac{1}{\sqrt{\gamma_0}}\right) (\bar{\partial}u_{l_1})\bar \partial v_{l_2}u_3 
\end{align*}
and 
\begin{align}\label{eq_m1_lastterm_part1}
 &  \int 4(T^1+iT^2)\partial \left(\frac{1}{\sqrt{\gamma_0}}\right)({\partial}u_{l_1})\bar \partial v_{l_2}u_3 \nonumber
\\&=  \nonumber \frac{4}{h^2} \int \left (F_{T,\gamma_0} e^{\frac{\lambda f - \overline{\lambda f} -2\Phi + 2\bar\Phi}{h}} \partial {\Theta_1}  \bar{\partial}\overline\Theta_2 {A\bar{B}C} + F_{T,\gamma_0} e^\frac{\lambda f - \overline{\lambda f}}{h} \partial \Theta_1 \bar{\partial} \overline\Theta_2 A\bar{B}\bar{C} \right) \\&  \nonumber 
+ \frac{4}{h} \int F_{T,\gamma_0}e^{\frac{\lambda f -\overline{\lambda f} -2\Phi+2\bar{\Phi}}{h}} \partial \Theta_1 A (\bar\partial \bar B)C + F_{T,\gamma_0} e^{\frac{\lambda f- \overline{\lambda f}}{h}}\partial \Theta_1 A(\bar{\partial} \bar B)\bar C \\& \nonumber
+ \frac{4}{h} \int F_{T,\gamma_0}e^{\frac{\lambda f -\overline{\lambda f} -2\Phi+2\bar{\Phi}}{h}}  (\partial A) (\bar\partial \overline \Theta_2) \bar BC + F_{T,\gamma_0} e^{\frac{\lambda f- \overline{\lambda f}}{h}}(\partial  A)(\bar\partial \overline \Theta_2)\bar B\bar C \\&
+ o(\frac{1}{h}) 
\end{align}
by comparing to the expansion of product in \eqref{v1v2v3}, while here we do not have the $AB\bar{C}$ and $ABC$ terms.
Again, by non-stationary phase, using that $||A||_{H^2_{loc}}, ||B||_{H^2_{loc}}, ||C||_{H^2_{loc}}$ are $o(\frac{1}{h})$, the whole integral in \eqref{eq_m1_lastterm_part1} is $o(\frac{1}{h})$. 

Finally, we compute the integral
\begin{align}\label{m=1 last term}
  &  \int 4(T^1-iT^2)\partial \left(\frac{1}{\sqrt{\gamma_0}}\right) (\bar{\partial}u_{l_1})\bar \partial v_{l_2}u_3 \\& =  
    \nonumber \frac{4}{h^2} \int F_{T,\gamma_0} (\bar\partial \overline{\Theta_1})(\bar\partial \overline{\Theta_2}) \overline{ABC}  \\& \nonumber  +  \frac{4}{h} C_{z_0} (\gamma_0^{-1} \left(\bar{a}|a|^2\right)\partial \left(\frac{1}{\sqrt{\gamma_0}}\right)  (T^1-iT^2) (\bar\partial \overline{\Theta_1})(\bar\partial \overline{\Theta_2}) e^{(\bar\Phi-{\Phi})/h}|_{z=z_0} \\& \nonumber
+\int \frac{1}{h} F_{T,\gamma_0} \big[
2(\bar\partial \overline{\Theta_1})\,(\bar\partial \overline{a})\overline{{a}}^2+ 2\bar{a}(\bar{\partial} \overline{\Theta_2})(\bar{\partial}\overline{a}) \bar{a}
\big]+  o(\frac{1}{h})
\end{align}

Combining the results in \eqref{eq_m1_first_result}, \eqref{eq_m1_first_principal}, \eqref{quant_m1_2}, \eqref{quant_m1_3}, \eqref{quant_m1_4}, \eqref{quant_m1_5},  \eqref{quant_m1_6} and \eqref{quant_m1_7}, \eqref{m=1 last term},  we get from \eqref{eq:second integral} that

\begin{align*}
0=&\frac{4}{h^2}\int F_{T,\gamma_0}H_{\Theta} a^3
+\frac{4}{h^2}\int F_{T,\gamma_0}H_{\Theta} a^2(r_1+s_1+t_1) \\
&\quad
+\frac{1}{h}\int F_{T,\gamma_0}\, H_{\Theta} g(a,a_0, b_0, c_0)
\\
&\quad
+\frac{4}{h} \sum_{i=1,2} \biggl\{  C_{z_0}
\left.
\left(
\gamma_0^{-3/2}\, a|a|^2 \,\bar\partial(T^1+iT^2)\,(\partial\Theta_i)^2
\,e^{(\Phi-\bar\Phi)/h} 
\right)
\right|_{z=z_0}
\\
&\quad
+ C_{z_0}
\left.
\left(
\gamma_0^{-3/2}\,\overline{a}|a|^2 \,\partial(T^1-iT^2)\,
(\bar\partial\overline{\Theta_i})^2
\,e^{(\bar\Phi-\Phi)/h}
\right)
\right|_{z=z_0}
\\
&\quad
+ C_{z_0}
\left.
\left(
\gamma_0^{-1}(a|a|^2)\,\bar\partial\!\left(\frac{1}{\sqrt{\gamma_0}}\right)
(T^1+iT^2)(\partial\Theta_i)^2 e^{(\Phi-\bar\Phi)/h}
\right)
\right|_{z=z_0}
\\
&\quad
+ C_{z_0}
\left.
\left(
\gamma_0^{-1}(\overline{a}|a|^2)\,\partial\!\left(\frac{1}{\sqrt{\gamma_0}}\right)
(T^1-iT^2)(\bar\partial\overline{\Theta_i})^2 e^{(\bar\Phi-\Phi)/h}
\right)
\right|_{z=z_0} \biggr\} \\&\quad + o\!\left(\frac{1}{h}\right)
\end{align*}
where $H(\Theta)$ is a smooth function of derivatives of $\Theta_1$ and $\Theta_2$, $g(a,a_0,b_0,c_0)$ is a polynomial of $a, a_0, b_0$ and $c_0$ and their derivatives. Here, we note that $C_{z_0}$ is the same in all the middle results we used, because it depends only on the phase of the exponential in the integral for which we apply the stationary phase approximation.

First, by letting $h\rightarrow 0$,  we obtain the integrals of $h^{-2}
$ order in the front adds up to zero. Now, focus on the remaining terms. We separate the non-oscillating integrals of order $h^{-1}$ and the evaluation terms  using the same trick as in \cite{GuillarmouTzou2011}: 
Since $\psi(p)\neq 0$ we may choose a sequence of $h_j\to 0$ such that 
\begin{align*}
&\Re\left(C_{z_0}
\left.
\left(
\gamma_0^{-3/2}\, a|a|^2 \,\bar\partial(T^1+iT^2)\,(\partial\Theta_1)^2
\,e^{(\Phi-\bar\Phi)/h_j}
\right)
\right|_{z=z_0}\right) \\ & 
+ \Re \left(C_{z_0}
\left.
\left(
\gamma_0^{-1}(a|a|^2)\,\bar\partial\!\left(\frac{1}{\sqrt{\gamma_0}}\right)
(T^1+iT^2)(\partial\Theta_1)^2 e^{(\Phi-\bar\Phi)/h_j}
\right)
\right|_{z=z_0} \right) = 1
\end{align*}
and another sequence $\tilde h_j\to 0$ 
such that 
\begin{align*}
&\Re\left(C_{z_0}
\left.
\left(
\gamma_0^{-3/2}\, a|a|^2 \,\bar\partial(T^1+iT^2)\,(\partial\Theta_1)^2
\,e^{(\Phi-\bar\Phi)/\tilde{h}_j}
\right)
\right|_{z=z_0}\right) \\ & 
+ \Re \left(C_{z_0}
\left.
\left(
\gamma_0^{-1}(a|a|^2)\,\bar\partial\!\left(\frac{1}{\sqrt{\gamma_0}}\right)
(T^1+iT^2)(\partial\Theta_1)^2 e^{(\Phi-\bar\Phi)/\tilde{h}_j}
\right)
\right|_{z=z_0} \right) = -1
\end{align*}
for all $j$. 
Adding the expansion with $h=h_j$ and $h=\tilde h_j$, we deduce that the non-oscillating integrals of order $h^{-1}$ add up to zero. Therefore, we obtain 
$$
\Re \left({\bar{\partial}(T^1+iT^2)}- \frac{\bar{\partial}{\gamma_0}}{{\gamma_0}}(T^1+iT^2)\right)=0
$$
in $\Omega$,
 since $C_{z_0}\neq 0$, $a(z_0)\neq 0$, and $\partial \Theta_1 (z_0)\neq 0$.

In addition, let $u_2, u_3$ be the same as before, but let
$$
u_1 = i\frac{1}{\sqrt{\gamma_0}}(e^{\Theta_1/h}(\tilde{a} + h a_0 + r_1)
      - \overline{e^{\Theta_1/h}(\tilde{a} + h a_0 + r_1)}
      + e^{\varphi_1/h} r_2) 
$$
where $\tilde{a}$ is a holomorphic function that is purely imaginary on $\Gamma_0$, which ensures in Lemma 4.2.1 in  \cite{GuillarmouTzou2011} the existence of $a_0$ that annihilates $u_1$ on $\Gamma_0$. Note that by multiplying $i$, the new $u_1$ is still real-valued.

With this choice of $u_1, u_2$ and $u_3$, following the same computation, note that where $\bar{A}$ appeared before would be replaced by $-\bar{A},$ we would obtain
$$
\Im \left({\bar{\partial}(T^1+iT^2)}- \frac{\bar{\partial}{\gamma_0}}{{\gamma_0}}(T^1+iT^2)\right)=0
$$
in $\Omega$.

Therefore, we get 
$$
{\bar{\partial}(T^1+iT^2)}- \frac{\bar{\partial}{\gamma_0}}{{\gamma_0}}(T^1+iT^2)=0
$$
in $\Omega$. By uniqueness of solution for the elliptic equation, we have $T^1+iT^2=0$ in $\Omega$. It follows that $T^1=T^2=0$ in $\Omega$ since we assume the coefficients are real valued.

\section{Proof of Proposition: $m= 2$}
For $m=2$, we start with
\begin{align}\label{eq:m2identity}
    \sum_{\left(l_1, l_2, l_3\right) \in \pi(3)} \sum_{j, k=1}^2 \int T^{j k}(x) \partial_{x_j} u_{l_1} \partial_{x_k} u_{l_2} \nabla u_{l_3} \cdot \nabla u_4 =0.
\end{align}

If we let two of the functions $u_1, u_2,u_{3}$  be constant functions equal to $1$, then we get
\begin{equation*}
    \int  T^{00}(x) \nabla u \cdot \nabla v  = 0 ,
\end{equation*}
which is the same as the identity \eqref{s3e1} in  the $m=1$ case for $T^{0}$. This proves $T^{0 0}=0$ in $\Omega$. Next, we let one of $u_1, u_2,u_{3}$ be the constant function $1$. This yields
\begin{equation*}
    \sum_{\left(l_1, l_2\right) \in \pi(2)} \sum_{j=1}^2 \int T^{0j}(x) \partial_{x_j} u_{l_1} \nabla u_{l_2} \cdot \nabla u_3  = 0,
\end{equation*}
which is the same identity \eqref{eq3} as in the
 $m=1$ case for $T^1$ and $T^2$. Thus we obtain  $T^{01}$ and  $T^{02}=0$ in $\Omega$.

 By using $T^{00}=T^{01}=T^{02}=0$,  the identity \eqref{eq:m2identity} becomes
\begin{align}\label{eqT^ij_neq0}
    \sum_{\left(l_1, l_2, l_3\right) \in \pi(3)} \sum_{j, k=1}^2 \int T^{j k}(x) \partial_{x_j} u_{l_1} \partial_{x_k} u_{l_2} \nabla u_{l_3} \cdot \nabla u_4 =0.
\end{align}
Next we choose solutions $u_l$, $l=1,\ldots, 4$, 

\begin{align*}
u_1 &= \frac{1}{\sqrt{\gamma_0}}(e^{\Theta_1/h}(a + h a_0 + r_1)
      + \overline{e^{\Theta_1/h}(a + h a_0 + r_1)}
      + e^{\varphi_1/h} r_2) \\
u_2 &= \frac{1}{\sqrt{\gamma_0}}(e^{\Theta_2/h}(a + h b_0 + s_1)
      + \overline{e^{\Theta_2/h}(a + h b_0 + s_1)}
      + e^{\varphi_2/h} s_2)\\
u_3 &=\frac{1}{\sqrt{\gamma_0}}( e^{\Theta_3/h} {({a} + h c_0 + t_1)}
      + \overline{e^{\Theta_3/h}({a} + h c_0 + t_1)}
      + e^{\varphi_3/h} t_2)\\
u_4 &=\frac{1}{\sqrt{\gamma_0}}( e^{\Theta_4/h} \overline{({a} + h d_0 + l_1)}
      + \overline{e^{\Theta_4/h}}({a} + h d_0 + l_1)
      + e^{\varphi_4/h} l_2)\\      
\end{align*}

For the first set of solutions, we let $\Theta_1, \Theta_2, \Theta_3, \Theta_4$ be
\begin{align*}
    \lambda f, \quad \lambda f, \quad\Phi-2\lambda f , \quad -\bar{\Phi}
\end{align*}
respectively, where $\lambda$ can be chosen the same as before, so that $\lambda f$, $\Phi -\lambda f$ has no critical point in $\Omega$
.

Like before, let $A=a+ha_0+r_1,B=a+hb_0+s_1, {C}={a+hc_0+t_1}, D = a+hd_0+ l_1$, we frist compute
\begin{align}\label{v1v2v3v4}
v_1v_2v_3v_4 \nonumber
&= \nonumber
2\operatorname{Re}\Bigl[
e^{(\Theta_1+\Theta_2+\Theta_3+\Theta_4)/h}ABC\overline D
+e^{(\Theta_1+\Theta_2+\Theta_3+\overline{\Theta_4})/h}ABCD
\\ \nonumber
&\qquad
+e^{(\Theta_1+\Theta_2+\overline{\Theta_3}+\Theta_4)/h}AB\overline C\,\overline D
+e^{(\Theta_1+\Theta_2+\overline{\Theta_3}+\overline{\Theta_4})/h}AB\overline C\,D
\\ \nonumber
&\qquad
+e^{(\Theta_1+\overline{\Theta_2}+\Theta_3+\Theta_4)/h}A\overline B C\overline D
+e^{(\Theta_1+\overline{\Theta_2}+\Theta_3+\overline{\Theta_4})/h}A\overline B C D
\\ \nonumber
&\qquad
+e^{(\Theta_1+\overline{\Theta_2}+\overline{\Theta_3}+\Theta_4)/h}A\overline B\,\overline C\,\overline D
+e^{(\Theta_1+\overline{\Theta_2}+\overline{\Theta_3}+\overline{\Theta_4})/h}A\overline B\,\overline C\,D
\Bigr]\\
&\qquad
+ \text{terms containing at least one of $r_2, s_2, t_2, l_2$}
\end{align}
By our choice of $\Theta_i'$ s, 
$$
\Theta_1+\Theta_2+\Theta_3+\Theta_4 = \Phi- \bar{\Phi}, \qquad \Theta_1+\Theta_2+\Theta_3+\overline{\Theta}_4 = 0,
$$
while all the other six phases in the expansion have no critical point. 

Therefore, comparing to \eqref{v1v2v3}, by a similar computation in the case $m=1$, we get 
\begin{equation}\label{m2result1}
T^{11}+2iT^{12}-T^{22}=0.
\end{equation}

To get other linear combinations of $T^{11}, T^{12}$ and $T^{22}$, we consider the second set of solutions, where we choose $\Theta_1, \Theta_2, \Theta_3, \Theta_4$ to be
\begin{align}\label{phases m = 2}
    \lambda f,  \Phi-\lambda f , \overline{\lambda f}, -\overline{\Phi}-\overline{\lambda f}
\end{align}
respectively. Here, we choose $\lambda$ and $f$ such that the phases all have no critical points, and are real on $\Gamma_0$.
To achieve this,  we first let $f$ be a conformal map that maps $\Gamma_0$ to part of the real line. Now as $f'\neq 0$ on $\Omega$, consider  
$$
r_1(z) = \frac{\Phi'(z)}{f'(z)}, \quad r_2(z) = -\frac{\Phi'(z)}{f'(z)}
$$
is holomorphic on $\Omega$,
with both images being some bounded subset of $\C$. Let $\lambda$ be chosen outside the image of $r_1$ and $r_2$, then all the phases have no critical point in the domain. Also, the phases are real on $\Gamma_0$ since $f$ and $\Phi$ are. 

Then we rewrite \eqref{eqT^ij_neq0} as 
\begin{multline}\label{eqT11+T22}  
%\end{multi}%\label{eq10}
    \sum_{\left(l_1, l_2, l_3\right) \in \pi(3)}  \int \Big[T^{11}((\partial+\bar{\partial}) u_{l_1})( (\partial+\bar{\partial}) u_{l_2})+ T^{12} ((\partial+\bar{\partial})u_{l_1})( i(\partial-\bar \partial){u_{l_2}} )\\
    + T^{21} (i(\partial-\bar{\partial})u_{l_1}) ((\partial+\bar \partial){u_{l_2}}) -T^{22} ((\partial-\bar \partial)u_{l_1} )((\partial-\bar \partial)u_{l_2} ) \Big] (\partial u_{l_3}\bar{\partial}u_4+ \bar{\partial} u_{l_3} \partial u_4)  =0.
\end{multline}
Similar to the expansion \eqref{v1v2v3v4}, if we focus on the terms without remainders $r_2, s_2, t_2$ and $l_2$, the eight phase combinations inside the real part become
\begin{align*}
\Theta_1+\Theta_2+\Theta_3+\Theta_4
&=\Phi-\overline{\Phi},\\
\Theta_1+\Theta_2+\Theta_3+\overline{\Theta_4}
&=\overline{\lambda f}-\lambda f,\\
\Theta_1+\Theta_2+\overline{\Theta_3}+\Theta_4
&=\Phi-\overline{\Phi}+\lambda f-\overline{\lambda f},\\
\Theta_1+\Theta_2+\overline{\Theta_3}+\overline{\Theta_4}
&=0,\\
\Theta_1+\overline{\Theta_2}+\Theta_3+\Theta_4
&=\lambda f-\overline{\lambda f},\\
\Theta_1+\overline{\Theta_2}+\Theta_3+\overline{\Theta_4}
&=\overline{\Phi}-\Phi,\\
\Theta_1+\overline{\Theta_2}+\overline{\Theta_3}+\Theta_4
&=\lambda f-\overline{\lambda f},\\
\Theta_1+\overline{\Theta_2}+\overline{\Theta_3}+\overline{\Theta_4}
&=\overline{\Phi}-\Phi+\lambda f-\overline{\lambda f}.
\end{align*}
We first consider the  integral term in the expansion of the left-hand side, where the phase in the integral vanishes, i.e. either 
$$
\Theta_1 + {\Theta}_2 +  \overline{\Theta}_3 +\overline{\Theta_4}
$$
or its conjugate. Note that now $\Theta_1, \Theta_2$ are holomorphic, $\Theta_3, \Theta_4$ are antiholomorphic. Hence, if exists, the $h^{-4}$ terms corresponds to terms
$$
 \partial u_1 \partial u_2 \partial u_3 \partial u_4 \qquad \text{or} \qquad \overline{\partial} u_1 \overline{\partial} u_2 \overline{\partial} u_3 \overline{\partial} u_4
$$
However, these terms do not exist due to the form of $(\partial u_{l_3}\bar{\partial}u_4+ \bar{\partial} u_{l_3} \partial u_4)$ in \eqref{eqT11+T22}.

Indeed, now the term in the expansion with zero phase are of $h^{-3}$ order, where the derivative on one solution does not hit the phase. 

Next, we look at the integral term in the expansion that has the phase with a nondegenerate critical point, i.e.
$$\Theta_1+\Theta_2+\Theta_3+\Theta_4, \qquad\Theta_1+\overline{\Theta_2}+\Theta_3+\overline{\Theta_4}$$
or their conjugates. If the derivatives all hit on the phase of the solution, by stationary phase we would get a $h^{-3}$ evaluation term with an exponential corresponding to 
$$
\partial u_1 \partial u_2 \bar{\partial}u_3 \bar{\partial}u_4 \qquad \text{or}  \qquad \bar{\partial}u_1\bar{\partial}u_2 \partial u_3 \partial u_4
$$
and 
$$
\partial u_1 \bar\partial u_2 \bar \partial u_3 {\partial}u_4 \qquad \text{or} \qquad \bar\partial u_1 \partial u_2  \partial u_3 \bar{\partial}u_4
$$
Now recall \eqref{eqT11+T22}, we observe that all the terms above have a common coefficient $T^{11}+T^{22}$.

Finally, we observe that all the other phase combinations give a phase with no critical point. By non-stationary phase, these terms give $o(h^{-3})$ integrals by integration by parts. 

Now, by the same trick that separates integrals with no exponential and evaluation terms with an exponential $e^{(\Phi-\bar{\Phi})/h}$ in the case $m=1$, we can extract the evaluation term of $h^{-3}$ order, and show
\begin{align}
    T^{11}+T^{22}=0
\end{align}
Combining with \eqref{m2result1}, we get 
$T^{11}, T^{22}$and $ T^{12}$ all vanishes in the domain.

\section{Proof of Proposition: $m\geq 3$ }

Same as in section 4 of \cite{LW24}, by induction we can recover $T^{0 j_1\cdots j_{m-1}}$ where
$j_1,\ldots, j_{m-1} \in \{0,1,2\}$.

It remains to prove $T^{j_1\cdots j_m}=0$ where all the indices $j_k$, $k=1,\ldots, m$, are nonzero. The integral identity is reduced to
\begin{align}\label{eq10}
    \sum_{\left(l_1, \ldots, l_{m+1}\right) \in \pi(m+1)} \sum_{j_1,\ldots j_m=1}^2 \int T^{j_1 \cdots j_m}(x) \partial_{x_{j_1}} u_{l_1} \partial_{x_{j_2}} u_{l_2} \ldots \partial_{x_{j_m}} u_{l_m} \nabla u_{l_{m+1}} \cdot \nabla u_{m+2} =0
\end{align} for all $u$ solving \eqref{cond_equ}. Note that indices in the inner sum now start from $1$.
Since $T^{j_1\cdots j_m}$ is symmetric in exchange of any of its two indices, we have $m+1$ unknown entries $T^{j_1\cdots j_m}$ in \eqref{eq10}. We begin by considering the special case $m=3$ in detail, and then the general case $m$, with some technical details omitted.

\subsection{Case $m=3$}
The integral identity now is
\begin{align}\label{eq10}
    \sum_{\left(l_1, \ldots, l_{4}\right) \in \pi(4)} \sum_{j_1,\ldots j_3=1}^2 \int T^{j_1 \cdots j_3}(x) \partial_{x_{j_1}} u_{l_1} \partial_{x_{j_2}} u_{l_2} \partial_{x_{j_3}} u_{l_3} \nabla u_{l_{4}} \cdot \nabla u_{5} =0
\end{align}
Like before, our solutions shall have the form 
\begin{align*}
u_i &= \frac{1}{\sqrt{\gamma_0}}(e^{\Theta_i/h}(a + h (a_0)_i + (r_1)_i)
      + \overline{e^{\Theta_i/h}(a + h (a_0)_i + (r_1)_i)}
      + e^{\varphi_i/h} (r_2)_i) 
\end{align*}
Now we consider two sets of solutions. In the following expression, we write the choice of $\Theta_1, \Theta_2, \Theta_3, \Theta_4, \Theta_5$ in order.\\

\noindent\textbf{Set I:} 
\begin{align*}
  \lambda f, \lambda f, \lambda f, \Phi-3\lambda f, -\bar{\Phi}
\end{align*}

\noindent\textbf{Set II:} 
\begin{align*}
  \lambda f, \lambda f,  \Phi-2\lambda f, \overline{\lambda f}, -\bar{\Phi} - \overline{\lambda f}
\end{align*}
Here in both sets we choose $\lambda$ such that 
$$
\lambda f, \quad \Phi \pm \lambda f, \quad \Phi \pm 2\lambda f, \quad \Phi \pm 3\lambda f
$$
all have no critical point in $\Omega$. This can be achieved by a similar argument as after \eqref{phases m = 2}.

The phases in the expansion result from adding the phase above or its conjugate in a row. Note that in both choices, the phases add up to $\Phi - \bar{\Phi}$, and replacing a phase by its conjugate in the summation is changing the original sum by a purely imaginary term, so the phases in the expansion are all purely imaginary. Moreover, by the choice of $\lambda$, they can be divided into three types: one with no exponential term (all phases cancels out in the summation), one with a phase with nondegenerate critical point, and the other with a phase with no critical points.
By a similar argument as before, using Set I and Set II of solutions, respectively, we could show
\begin{align*}
T^{111}+3iT^{112}-3T^{122}-iT^{222}&=0\\
    T^{111}-iT^{112}+T^{122}-iT^{222}&=0
\end{align*}
which implies all the entrances vanish (since we also assume all the entrances are real).
\subsection{Case for general $m$}
We will choose $\lfloor \frac{m}{2}\rfloor+1$ set of solutions for case $m$. Each set shall have $m+2$ solutions, as the integral identity indicates.
 
Again, solutions have the form
\begin{align*}
u_i &= \frac{1}{\sqrt{\gamma_0}}(e^{\Theta_i/h}(a + h (a_0)_i + (r_1)_i)
      + \overline{e^{\Theta_i/h}(a + h (a_0)_i + (r_1)_i)}
      + e^{\varphi_i/h} (r_2)_i) 
\end{align*}
with the following choice of phases in the $j-th$ set, $j = 1,2,\ldots, \lfloor \frac{m}{2}\rfloor+1$:

\noindent\textbf{Set j:} 
\begin{align*}
    &\text{For } i=1,2, \dots,m+1-j:\quad \Theta_i = \lambda f \\
    & \Theta_{m+2-j} = \Phi - (m+2-j)\lambda f \\
    & \text{For } i=m+3-j, \dots, m+1: \quad \Theta_i = \overline{\lambda f}\\
    & \Theta_{m+2} = -\bar{\Phi} - (j-1)\overline{\lambda f}
\end{align*}
We shall choose $\lambda$ such that $\lambda f$ and all of $\Phi \pm k\lambda$, $k=1,2,\dots, m$ has no critical point in $\Omega$. For fixed $m$, there are finitely many conditions, and they can be achieved by a similar argument as in previous cases.

By the same proof in section 4 of \cite{LW24},
for $j=1$, we could obtain by the first set of solutions 
\begin{align}\label{eq:coefficientline1}
   \binom{m}{0} T^{1\cdots1}+i\binom{m}{1}T^{1\cdots12}+i^2 \binom{m}{2}T^{1\cdots122}+\cdots+i^{m}\binom{m}{m}T^{2\cdots2}&=0 \quad;
\end{align}
for $j =2, \dots, \lfloor \frac{m}{2}\rfloor+1$, we would obtain
\newcommand{\Tsym}[2]{T^{\underbrace{1\cdots 1}_{#1}\,\underbrace{2\cdots 2}_{#2}}}

\begin{equation}\label{eq:full_linear_system}
\sum_{s=0}^{m} C_j(s)\;\Tsym{s}{m-s} \;=\; 0,
\qquad j=2,3,\dots, \lfloor \frac{m}{2}\rfloor+1,
\end{equation}
where the coefficients are
\begin{align*}\label{eq:coeffs}
C_j(s)
&= \sum_{\substack{p=0,\dots,j-1\\ q=0,\dots,m+2-j-1\\ p+q=m-s}}
\binom{j-1}{p}\binom{m+2-j-1}{q}\, i^{p}\,(-i)^{q}.
\end{align*}
Note that since all $T$ coefficients are real, by taking the conjugate of each equation, we will get the same equation system as in section 4 of \cite{LW24}. Hence, by the same argument as in the above paper, we show that we have a linear independent system of equations, and so every $T$ coefficient vanishes. 

\bibliographystyle{alpha}
\bibliography{ref}

\end{document}